%% file: main.tex
\documentclass[12pt]{amsart}

\input{style}

\input{macros}

\usepackage{rotating}

\title{Model theory of valued fields with an endomorphism}

\author{Simone Ramello}

\begin{document}

\maketitle

\begin{abstract}
    We establish relative quantifier elimination for valued fields of residue characteristic zero enriched with a non-surjective valued field endomorphism, building on recent work of Dor and Halevi. In particular, we deduce relative quantifier elimination for the limit of the Frobenius action on separably closed valued fields of positive characteristic.
\end{abstract}

\section{Introduction}
\input{introduction}
\section*{Acknowledgements}
First and foremost, I would like to thank my supervisors Martin Hils and Franziska Jahnke for their invaluable support, ideas, guidance, and patience as this project evolved, hit walls, and overall led us on an emotional rollercoaster. I would also like to thank the Parisian crowd, who hosted me for several wonderful weeks during which some of the core work of this paper was done: in particular Sylvy Anscombe, Zoé Chatzidakis, and Silvain Rideau-Kikuchi, three people whose deep insights are all (hopefully) fully and faithfully represented in this work. I would also like to thank Margarete Ketelsen for being there to answer my questions (even when they were shouted from one office to the other) and Rosario Mennuni for endless discussions over voice messages. Special thanks go to Yuval Dor for an insightful (and motivating!) Zoom chat around contracting valued difference fields, which contained many questions that only many months later I would realize were fundamental, and for both pointing out a serious mistake in a previous version of this manuscript and extensively discussing its fix. I'd also like to thank a variety of other people with whom I had discussions on this topic over the years: Blaise Boissonneau, Martin Bays, Anna De Mase, Stefan Ludwig, Tom Scanlon, Pierre Touchard, Mariana Vicaria, Tingxiang Zou.

The author has been supported by the Deutsche Forschungsgemeinschaft (DFG, German Research Foundation) under the Excellence Strategy EXC 2044–390685587, \textit{Mathematics Münster: Dynamics–Geometry–Structure}, via HI 2004/1-1 (part of the French-German ANR-DFG project \textit{GeoMod}), and via the DFG project 495759320 -- \textit{Modelltheorie bewerteter Körper mit Endomorphismus}.
\section{Preliminaries}\label{preliminaries}
\input{preliminaries}
\section{$\sigma$-henselianity and immediate extensions}\label{henselianity}
\input{henselianity-and-kaplansky}
\section{Setting up the embedding lemma}\label{auxiliary}
\input{auxiliary-steps}
\section{The embedding lemma and Ax-Kochen/Ershov}\label{AKE}
\input{embedding-and-consequences}
\subsection{$\mathrm{NTP}_2$}
\input{advanced-mt}
\bibliographystyle{alpha}
\bibliography{biblio}
\end{document}

%% file: style.tex
\usepackage{graphicx}
\usepackage{amsmath,amssymb,mathtools,amsthm,amsfonts}
\usepackage{dsfont}
\usepackage[osf]{mathpazo}
\usepackage{eucal,mathabx}
\usepackage{todonotes}
\usepackage[new]{old-arrows}

\renewcommand\to{\mathchoice{\longrightarrow}{\rightarrow}{\rightarrow}{\rightarrow}}
\newcommand*{\longtwoheadrightarrow}{\ensuremath{\relbar\joinrel\twoheadrightarrow}}
\usepackage{fullpage}
\usepackage{hyperref}
\newcommand{\runin}[1]{\subsection{#1}}
\usepackage{faktor,quiver}
\usepackage[capitalise]{cleveref}

\usepackage{xcolor}
\definecolor{special}{HTML}{4e8cbe}
\renewcommand{\leq}{\leqslant}
\renewcommand{\geq}{\geqslant}

\newtheoremstyle{main}
{}
{}
{\itshape}
{0pt}
{}
{}
{ }
{\textbf{\thmname{#1} \thmnumber{#2}}\thmnote{{ {(#3)}}}\textbf{.}}

\theoremstyle{main}
\newtheorem{mtheorem}{Theorem}

\newtheoremstyle{def-custom}
{}
{}
{}
{0pt}
{}
{}
{ }
{\textbf{\thmname{#1} {\thmnumber{#2}}}\thmnote{ (\textit{#3})}\textbf{.}}

\newtheoremstyle{def-custom-it}
{}
{}
{\itshape}
{0pt}
{}
{}
{ }
{\textbf{\thmname{#1} {\thmnumber{#2}}}\thmnote{ (\textit{#3})}\textbf{.}}

\theoremstyle{def-custom-it}
\newtheorem{theorem}{Theorem}[section]
\newtheorem{proposition}[theorem]{Proposition}
\newtheorem{corollary}[theorem]{Corollary}
\newtheorem{lemma}[theorem]{Lemma}
\newtheorem{fact}[theorem]{Fact}
\theoremstyle{def-custom}
\newtheorem{definition}[theorem]{Definition}
\newtheorem{remark}[theorem]{Remark}
\newtheorem{example}[theorem]{Example}

\newtheorem{assumption}[theorem]{Assumption}

%% file: macros.tex
\newcommand{\RV}{\mathrm{RV}}
\newcommand{\rv}{\operatorname{rv}}
\newcommand{\res}{\operatorname{res}}
\newcommand{\val}{v}
\newcommand{\residue}[1]{k_{#1}}
\newcommand{\valuegroup}[1]{\Gamma_{#1}}
\newcommand{\ac}{\operatorname{ac}}
\renewcommand{\O}{\mathcal{O}}
\newcommand{\m}{\mathfrak{m}}
\newcommand{\Wit}{\mathrm{W}}
\newcommand{\smax}[1]{\overline{#1}}
\newcommand{\invmax}[1]{\smax{#1}^{\mathrm{inv}}}
\newcommand{\pseudoconverges}{\implies}
\newcommand{\completion}[1]{\widehat{#1}}

\newcommand{\Frac}{\operatorname{Frac}}

\newcommand{\Res}[1]{\overline{#1}}
\newcommand{\Val}[1]{{#1}_{\Gamma}}
\newcommand{\Rv}[1]{{#1}_{\rv}}
\newcommand{\End}{\operatorname{End}}
\newcommand{\primary}{\mathrm{FE}}
\newcommand{\FE}[1]{\mathrm{FE}(#1)}
\newcommand{\inv}[1]{#1^{\mathrm{inv}}}
\newcommand{\core}[1]{{#1}_{\mathrm{inv}}}
\newcommand{\ext}[2]{#1\langle#2\rangle}
\newcommand{\extnot}[2]{#1[#2\rangle}
\newcommand{\polring}[1]{#1[X]_{\sigma}}
\newcommand{\shifted}[1]{\mathcal{S}(#1)}

\renewcommand{\L}{\mathfrak{L}}
\newcommand{\WTH}{\mathrm{hVFE}}
\newcommand{\WTHzero}{\WTH_{0}}
\newcommand{\WTHzeroac}{\WTHzero^{\mathrm{ac}}}

\newcommand{\WTHzerosl}{\WTHzero^{s,\iota}}
\newcommand{\Tsep}{\primary^{\lambda}}
\newcommand{\Lringsigma}{\L_{\mathrm{ring},\sigma}}
\newcommand{\Loagsigma}{\L_{\mathrm{oag},\sigma}}
\newcommand{\Tvdf}{\mathrm{VFE}}
\newcommand{\Tvdfzero}{\mathrm{VFE}_0}
\newcommand{\Lvfe}{\L_{\mathrm{vfe}}}
\newcommand{\Lsep}{\Lvfe^{\lambda}}
\newcommand{\Lsepenrich}{\Lvfe^{\lambda,+}}

\newcommand{\LK}{\Lringsigma}
\newcommand{\LKsep}{\LK^{\lambda}}
\newcommand{\Lsorts}{\L_{\RVsort,\valuesort}}
\newcommand{\Lsepac}{\Lvfe^{\lambda,\mathrm{ac}}}
\newcommand{\Lseclift}{\Lvfe^{\lambda,s,\iota}}

\newcommand{\denseacvf}{\mathrm{ACVF}_{\mathrm{dense}}}
\newcommand{\LDP}{\L_{\mathrm{DP}}}
\newcommand{\LDPlambda}{\L^*_{\mathrm{DP}}}
\newcommand{\Lringpairs}{\L_{\mathrm{ring,P}}^{\lambda}}
\newcommand{\Th}{\operatorname{Th}}

\newcommand{\sort}[1]{\mathbf{#1}}
\newcommand{\model}[1]{#1}
\newcommand{\RVsort}{\sort{RV}}
\newcommand{\mainsort}{\sort{VF}}
\newcommand{\residuesort}{\sort{k}}
\newcommand{\valuesort}{\sort{\Gamma}}

\newcommand{\id}{\mathrm{id}}
\renewcommand{\mathbb}[1]{\mathds{#1}}

\newcommand{\nonsurj}{${\lozenge}$}

%% file: introduction.tex
\runin{Valued fields}
For the model-theoretically inclined, understanding valued fields has historically been a game of finding the correct invariants for their theories, reducing the possibly complicated valuational structure to something (at least a priori) easier to understand and classify. Given a valued field $(K,\val)$, one can naturally look at its value group $\valuegroup{K}$ and residue field $\residue{K}$ as possible (model-theoretic) invariants, in the sense that one could hope to determine the theory of $(K,\val)$ using the theories of $\valuegroup{K}$ and $\residue{K}$. This concretized in the celebrated Ax-Kochen/Ershov results (\cite{ax1965diophantine}, \cite{ershov1965elementary}), which establish that -- when $(K,\val)$ is henselian, and both $K$ and $\residue{K}$ have characteristic zero -- the theory of $(K,\val)$ is implied by the theories of $\valuegroup{K}$ (as an ordered group) and $\residue{K}$ (as a field).

\runin{Endomorphisms of valued fields} A great amount of work has since then blossomed, which seeks to generalize similar results in various directions. We tackle one such possible way of making the question harder, namely we enrich the structure of $(K,\val)$ with a valued field endomorphism $\sigma$. Beyond the abstract interest for such structures, we note that they arise naturally in several interesting contexts: for example, if we let $(K_p,\val_p)$ be either a separably closed or an algebraically closed valued field of characteristic $p$ for every $p$, then we can expand one language of valued fields of our choice to be able to consider the structures $(K_p, \val_p, \phi_p)$, where $\phi_p$ is a unary function that we interpret to be $\phi_p(x) = x^p$. If we take the ultraproduct along some non-principal ultrafilter $\mathcal U$ over the set of prime numbers, then, we end up with an algebraically closed valued field of residue characteristic zero equipped with some (possibly non-surjective) endomorphism that acts as an infinite non-standard Frobenius. These structures are now completely understood, first by work of Durhan\footnote{Published under the surname Azgın.} (\cite{azgin2010valued}), and then later by work of Dor and Hrushovski (\cite{dor2022specialization}) and Dor and Halevi (\cite{dor2023contracting}), building on the twisted Lang-Weil estimates (\cite{nonstandardfrob}, \cite{shuddhodan2021hrushovski}). These endomorphisms of valued fields are of a special kind, namely \textit{$\omega$-increasing}. At the other end of the spectrum sits the Witt-Frobenius, whose model theory we also understand thanks to \cite{belair2007model} and \cite{azgin2011elementary}. There, the endomorphism is surjective, so it is an automorphism; moreover, it is not just an automorphism of valued fields, but it is in fact valuation preserving, i.e. an \textit{isometry}. In full generality, Durhan and Onay in \cite{onay2013quantifier} prove a relative quantifier elimination for a natural class of valued fields with an automorphism, showing that their theories can be understood in terms of their leading terms structures $\RV$ together with the induced automorphism.

\runin{Beyond surjectivity} The most recent work of Dor and Halevi (\cite{dor2023contracting}) takes a step in the direction of removing the assumption of surjectivity of the endomorphism in question. They do so in the \textit{absolute} case, namely identifying the model companion for valued fields enriched with an $\omega$-increasing endomorphism. In turn, their work is based on the seminal results of Chatzidakis and Hrushovski (\cite{chatzidakis2004model}), which isolate the model companion for fields with an endomorphism. Building on their work, and that of Durhan and Onay, we remove the surjectivity assumption in the \textit{relative} case, proving a relative quantifier elimination down to the leading term structure.
\par On the difference side, following \cite{dor2023contracting} we need to work with models of $\primary$, namely difference fields $(K,\sigma)$ where $\sigma(K)$ is relatively algebraically closed in $K$. This is a technical assumption, which is clarified in \cref{FE-explained}.

\begin{mtheorem}[{\cref{embedding-theorem}}]
    Let $\WTHzero$ be the theory of non-inversive weakly $\sigma$-henselian valued difference fields $(K,\val,\sigma)$ of residue characteristic zero, with $\sigma(K)$ relatively algebraically closed in $K$. Then, modulo $\WTHzero$, every formula is equivalent to one where the quantifiers only range over variables from $\RVsort$ and $\valuesort$.
\end{mtheorem}

\runin{Consequences} Once relative quantifier elimination using $\RV$ is established, one can immediately deduce a series of transfer theorems for models of $\WTHzero$ in a variety of languages.

\begin{mtheorem}[\cref{relative-qe-ac} and \cref{relative-qe-sl}]
    Let $\WTHzeroac$ (resp. $\WTHzerosl$) be the theory of non-inversive weakly $\sigma$-henselian valued difference fields $(K,\val,\sigma)$ of residue characteristic zero, with $\sigma(K)$ relatively algebraically closed in $K$, which come equipped with a $\sigma$-equivariant angular component (resp. a section of the valuation and a lift of the residue field). Then, modulo $\WTHzeroac$ (resp. $\WTHzerosl$), every formula is equivalent to one where the quantifiers only range over variables from $\residuesort$ and $\valuesort$.
\end{mtheorem}

From this we can deduce a transfer theorem for $\mathrm{NTP}_2$ along the strategy introduced in \cite{chernikov2014valued}. In particular, this gives a different proof of \cite[Corollary 8.6]{dor2023contracting}, namely that the theory that \cite{dor2023contracting} call $\widetilde{\mathrm{VFE}}_0$ is $\mathrm{NTP}_2$.

\begin{mtheorem}[\cref{ntp2-rv}]
    Let $(K,\val,\sigma) \vDash \WTHzero$. Then $(K,\val,\sigma)$ is $\mathrm{NTP}_2$ if and only if $(\RV_K,\Rv{\sigma})$ is $\mathrm{NTP}_2$.
\end{mtheorem}

\runin{The road ahead} The paper is structured as follows.
\begin{itemize}
    \item \cref{preliminaries} introduces the necessary preliminaries in difference algebra and valued difference algebra, in particular focusing on the notion of $\sigma$-separability.
    \item \cref{henselianity} develops the theory of weak $\sigma$-henselianity and immediate extensions, proving the crucial \cref{separable-kaplansky}.
    \item \cref{auxiliary} takes care of the other major ingredient for an embedding lemma, namely how to extend value (difference) group and residue (difference) field.
    \item \cref{AKE} proves the embedding lemma and deduces a wealth of variants and model-theoretic consequences, including $\mathrm{AKE}$-type theorems and $\mathrm{NTP}_2$ transfer.
\end{itemize}

%% file: preliminaries.tex
Unless otherwise stated, all fields considered will be of characteristic zero.

\runin{Notation} If $\model{M}$ is an $\L$-structure, $\sigma \in \End(\model{M})$, write $\L_\sigma$ for $\L \cup \{\sigma\}$. We write $\inv{\model{M}}$ for the unique (up to isomorphism over $M$) $\L_\sigma$-structure given by $\inv{M} \coloneqq \bigcup_{n \geq 0} \sigma^{-n}(M)$. We write $\core{\model{M}}$ for the unique $\L_\sigma$-structure given by $\core{M} \coloneqq \bigcap_{n \geq 0} \sigma^{n}(M)$.

If $(\model{M},\sigma) \subseteq (\model{N},\sigma)$ as $\L_\sigma$-structures, then for any $\alpha \in N$ we write
    \[
        \extnot{M}{\alpha} \coloneqq \langle M, \{\sigma^n(\alpha)\}_{n\geq 0} \rangle_{\L} = \langle M, \alpha \rangle_{\L_\sigma} \subseteq N,
    \]
    and, if possible,
    \[
        \ext{M}{\alpha} \coloneqq \langle M, \{\sigma^n(\alpha),\sigma^{-n}(\alpha)\}_{n\geq 0} \rangle_{\L} \subseteq N.
    \]
We will often call $\model{M}$ \emph{inversive} if $\sigma$ is surjective.
\par If $K$ is a field, we denote by $K^{\mathrm{alg}}$ its algebraic closure, and by $K^{\mathrm{sep}}$ its separable closure. We reserve $\smax{K}$ for something else (see \cref{unique-maximal-def}). If $K$ is a field and $\sigma \in \End(K)$, then we call $(K,\sigma)$ a \emph{difference field}. Similarly, if $\Gamma$ is an ordered group and $\sigma \in \End(\Gamma)$, we call $(\Gamma,\sigma)$ an \emph{ordered difference group}.
\par If $(K,\val)$ is a valued field, we write $\O_K = \{x \in K \mid \val(x) \geq 0\}$ for its valuation ring, $\m_K = \{x \in K \mid \val(x) > 0\}$ for its maximal ideal, $\residue{K}$ for its residue field, $\valuegroup{K}$ for its value group, and $\RV_K$ for its leading terms structure. We denote by $\completion{K}$ its completion.

By a \emph{valued difference field}, or valued field with endomorphism, we mean a valued field $(K,\val)$ together with $\sigma \in \End(K,\val)$. For any such $\sigma \in \End(K,\val)$, we denote by $\Val{\sigma}$ the induced endomorphism of $\valuegroup{K}$, by $\Res{\sigma}$ the induced endomorphism of $\residue{K}$, and by $\Rv{\sigma}$ the induced endomorphism of $\RV_K$. We call $(\residue{K},\Res{\sigma})$ the \emph{residue difference field} of $(K,\val,\sigma)$, $(\valuegroup{K},\Val{\sigma})$ the \emph{value difference group} of $(K,\val,\sigma)$, and $(\RV_K,\Rv{\sigma})$ the \emph{leading terms difference structure} of $(K,\val,\sigma)$.

If $a \in K$ and $\gamma \in \valuegroup{K}$, we write
\[
    B_\gamma(a) = \{b \in K \mid \val(a-b) > \gamma\}
\]
for the \emph{open ball around $b$ of radius $\gamma$}, and
\[
    B_\gamma[a] = \{b \in K \mid \val(a-b) \geq \gamma\}
\]
for the \emph{closed ball around $b$ of radius $\gamma$}.

    Given a valued difference field $(K,\val,\sigma)$, we can see its value difference group $(\Gamma,\Val{\sigma})$ as a $\mathbb{Z}[\Val{\sigma}]$-module in a natural way. Given any $I = i_0 + i_1\Val{\sigma} + \dots + i_n\Val{\sigma}^n \in \mathbb{Z}[\Val{\sigma}]$, we write $\vert I \vert = i_0 + \cdots + i_n \in \mathbb{Z}$. Now, for any $\gamma \in \valuegroup{K}$, we let
\[
    I(\gamma) \coloneqq i_0 \cdot \gamma + i_1 \cdot \Val{\sigma}(\gamma) + i_2 \cdot \Val{\sigma}^2(\gamma) + \cdots + i_n \cdot \Val{\sigma}^n(\gamma) \in \valuegroup{K}.
\]

\begin{assumption}
    Unless otherwise stated, whenever we write $(K,\val,\sigma)$ we assume that both $\Val{\sigma}$ and $\Res{\sigma}$ (or, equivalently, $\Rv{\sigma}$) are surjective. Results that do not require this assumption are flagged by \nonsurj.
\end{assumption}

\runin{Difference polynomials}
We say that $p(X)$ is a \emph{difference polynomial} if there is a multivariable polynomial $P(X_0, \dots X_n)$ such that $p(X) = P(X,\sigma(X), \dots \sigma^n(X))$. We write
    \[
        p(X) = \sum_{I} a_{I} X^{I},
    \]
where for every $I = (i_0, \dots i_n) \in \mathbb{N}^{n+1}$, $X^I = X^{i_0}\sigma(X)^{i_1}\cdots\sigma^n(X)^{i_n}$.

If $(K,\sigma)$ is a difference field, then by $\polring{K}$ we denote the ring of difference polynomials with coefficients from $K$. If $p(X) = \sum_{I} a_{I} X^{I}$ is a difference polynomial with coefficients in $\O_K$, we write $\res(p)(X)$ for the difference polynomial $\sum_{I} \res(a_{I}) X^{I}$ over $\residue{K}$. Given any $p(X) = \sum_{I} \alpha_{I} X^{I} \in \polring{\residue{K}}$, an \emph{exact lift} of $p(X)$ is a difference polynomial $q(X) = \sum_{I} a_{I} X^{I} \in \polring{\O_{K}}$ such that $\res(a_{I}) = \alpha_{I}$, and if $\alpha_{I} = 0$ then $a_{I} = 0$.

Let $p(X)$ be a non-constant difference polynomial, say $p(X) = P(X,\sigma(X), \dots \sigma^n(X))$, where $X_n$ appears in $P$. Then, the \emph{complexity} $I \in \mathbb{N}^3$ of $p(X)$ is given by
\[
    I = (n,\deg_{X_n}(P), \operatorname{totdeg}(P)),
\]
where $\operatorname{totdeg}(P)$ is the total degree of $P$. If $p(X)$ is constant and non-zero, we say that it has complexity $(-\infty, 0, 0)$. The zero difference polynomial has complexity $(-\infty, -\infty, -\infty)$. We order complexities lexicographically, and declare that 
\[
    (-\infty,-\infty,-\infty) < (-\infty,0,0) < \mathbb N^3.
\]

Note that for any multivariable polynomial $P(X_0, \dots X_n)$ over $K$, there are unique polynomials $P_{J}(X_0, \dots X_n)$ such that, computing in $K[X_0, \dots X_n, Y_0, \dots Y_n]$,
\[
    P(\overline{X}+\overline{Y}) = \sum_{J \in \mathbb{N}^{n+1}} P_{J}(\overline{X}) Y_0^{j_0} \cdots Y_n^{j_n}.
\]
If $I = (i_0, \dots i_n) \in \mathbb{N}^{n+1}$, we let $I! \coloneqq i_0! \cdots i_n!$. Then, $I! P_{I}(X_0, \dots X_n) = \frac{\partial^{I}P}{\partial X^{I}}(X_0, \dots X_n)$. We define derivatives of difference polynomials by $p_{J}(X) \coloneqq \frac{\partial^{J} P}{\partial X^{J}}(X,\sigma(X), \cdots \sigma^n(X))$.

For any $j = 0, \dots n$, if $E_j = (\delta_{i,j})_{i=0}^{n}$, we write $p_j(X) = p_{E_j}(X)$. We often write $p'$ for $p_0$. Note that $p' = 0$ precisely when the variable $X$ does not appear in the difference polynomial, a fact that will guide us throughout this paper.

\runin{Shifting and changing variables}\label{shift-reference}
We devote a bit of space to the tedious operation of rigorously shifting difference polynomials back and forth using $\sigma$, so that we may keep track of how their derivatives change for later use.

First, we define the shift on coefficients. Computing in $\polring{\inv{K}}$, if $\ell \in \mathbb{Z}$, we define $p^{\sigma^\ell}(X) \coloneqq \sum_{I} \sigma^\ell(a_{I}) X^{I}$; then one has that, for any $\ell \in \mathbb{Z}$ and $I \in \mathbb{N}^{n+1}$, $(p^{\sigma^\ell})_{I}(X) = (p_{I})^{\sigma^\ell}(X)$.
Then, we define a change of variables that allows us to transform a difference polynomial $p$ into another difference polynomial $q$ such that $q' \neq 0$. Namely, if $\sigma^m(X)$ is the smallest iterate of $\sigma$ appearing in $p$, we operate a change of variables $Y \coloneqq \sigma^m(X)$.

\begin{definition}
For any non-constant $p(X)$ over $K$, let $m = m(p) \geq 0$ be the smallest natural number such that there is $I = (i_0, \dots i_n)$ with $i_m \neq 0$ and $p_{I} \neq 0$. Equivalently, $m(p)$ is the smallest such that $p_m \neq 0$.\footnote{In other words, $m$ is the smallest number of iterations of $\sigma(X)$ that appears in $p$.} Then, we let $\shifted{p}$ be the difference polynomial over $K$ defined by $\shifted{p}(Y) \coloneqq p(\sigma^{-m(p)}(Y))$.
\end{definition}

In a slight abuse of notation, if $I \in \mathbb{N}^{n+1}$, then by $\shifted{p_{I}}(Y)$ we mean again $p_{I}(\sigma^{-m}(Y))$ (and not, as one might imagine, $p_{I}(\sigma^{-m(p_{I})}(Y))$).
We note that $\shifted{p}^{\sigma^{-m}} = \shifted{p^{\sigma^{-m}}}$. 
We now compute the derivatives of $S(p)$ in terms of the derivatives of $p$. This is the key step in the computations necessary to prove \cref{the-horror}. We first define an operation on multi-indices.

\begin{definition}
Given any $I = (i_0, \cdots i_n) \in \mathbb{N}^{n+1}$, we let $I^{-m} = (i_{m}, \cdots i_n) \in \mathbb{N}^{n-m+1}$. 
Viceversa, for every $J = (j_0, \cdots j_{n-m}) \in \mathbb{N}^{n-m+1}$, let $J = (0, \cdots 0, j_0, \cdots j_{n-m}) \in \mathbb{N}^{n+1}$.
\end{definition}

For example, for every $m \leq j \leq n$, $E_j^{-m} = E_{j-m}$ and $E_j^{+m} = E_{j+m}$. Moreover, one has that $(I+J)^{\pm m} = I^{\pm m} + J^{\pm m}$, where $+$ is the pointwise sum on $\mathbb{N}^{n+1}$. By comparing Taylor expansions, we then obtain that, for $m = m(p)$ as above, 
\[
    \shifted{p}_{J}(Y) = \shifted{p_{J^{+m}}}(Y).
\]

Given any $a \in K$, we can thus compute (possibly inside $\inv{K}$), for every $I \in \mathbb{N}^{n+1}$,
\[
    \sigma^{-m}(p_{I}(a)) = p_{I}^{\sigma^{-m}}(\sigma^{-m}(a)) = \shifted{p_{I}}^{\sigma^{-m}}(a) = \shifted{p}^{\sigma^{-m}}_{I^{-m}}(a),
\]

and viceversa $p_{I}(a) = \sigma^m(\shifted{p}^{\sigma^{-m}}_{I^{-m}}(a))$.
In particular, for $m \leq j \leq n$, $p_j(a) = \sigma^m(\shifted{p}^{\sigma^{-m}}_{j-m}(a))$.

\begin{remark}

    Note that for any $b \in K$ one has
    \[
        p(b) = 0 \iff \sigma^{-m}(p(b)) = 0 \iff p^{\sigma^{-m}}(\sigma^{-m}(b)) = 0 \iff \shifted{p}^{\sigma^{-m}}(b) = 0,
    \]
    in other words $p$ and $\shifted{p}^{\sigma^{-m}}$ have the same zeroes in $K$, with the advantage that now $\shifted{p}$ and $\shifted{p}^{\sigma^{-m}}$ satisfy that $(\shifted{p})' \neq 0$ and $(\shifted{p}^{\sigma^{-m}})' \neq 0$.
    
\end{remark}

\runin{$\sigma$-separability} Drawing heavily from \cite[Section 4]{dor2023contracting}, which in turn builds on \cite{chatzidakis2004model}, we establish the corresponding notion to \textit{separability} for difference field extensions. One should think of a non-inversive difference field as an imperfect positive characteristic field, and indeed most of the basic facts from separability transfer to our setting, as long as we assume that we work with a model of $\primary$.

\begin{definition}
    Let $(K,\sigma) \subseteq (L,\sigma)$ be a difference field extension. We say that the extension is \emph{$\sigma$-separable} if $K$ is linearly disjoint from $\sigma(L)$ over $\sigma(K)$. We say that the extension is \emph{almost $\sigma$-separable} if $K$ is algebraically free from $\sigma(L)$ over $\sigma(K)$. We say that the extension is \emph{$\sigma$-separably $\sigma$-algebraic} if it is $\sigma$-separable and $\sigma$-algebraic.
\end{definition}

\begin{remark}
    With an induction argument, one sees that $K$ is linearly disjoint from $\sigma(L)$ over $\sigma(K)$ if and only if $\inv{K}$ is linearly disjoint from $L$ over $K$.
\end{remark}

\begin{fact}[{\cite[Propositions 4.5, 4.29, 4.30, 4.31]{dor2023contracting}}]\label{separability-cheat-sheet}
    Let $(K,\sigma) \subseteq (M,\sigma) \subseteq (L,\sigma)$ be a tower of difference fields.
    \begin{enumerate}
        \item If $(K,\sigma) \subseteq (M,\sigma)$ and $(M,\sigma) \subseteq (L,\sigma)$ are $\sigma$-separable (resp. almost $\sigma$-separable), then $(K,\sigma) \subseteq (L,\sigma)$ is $\sigma$-separable (resp. almost $\sigma$-separable),
        \item If $(K,\sigma) \subseteq (L,\sigma)$ is $\sigma$-separable (resp. almost $\sigma$-separable), then $(K,\sigma) \subseteq (M,\sigma)$ is $\sigma$-separable (resp. almost $\sigma$-separable).
        \item Let $K_0$ be a difference extension of $K$, linearly disjoint from $L$ over $K$. Then, $(K,\sigma) \subseteq (L,\sigma)$ is $\sigma$-separable (resp. almost $\sigma$-separable) if and only if $(K_0,\sigma) \subseteq (L\otimes_K K_0,\sigma)$ is $\sigma$-separable (resp. almost $\sigma$-separable).
    \end{enumerate}
\end{fact}

\runin{The theory $\primary$}\label{FE-explained} Here, we diverge from the intuition from the positive characteristic world: $\sigma$-separability does not behave well unless we assume that $K$ is relatively algebraically closed in $\inv{K}$. This is, at least partially, because we wish to be able to check $\sigma$-separability of a $\sigma$-algebraic extension by computing the first derivatives of the difference polynomials that witness algebraicity; however, ordinary minimal polynomials happen to also be difference polynomials, and since we work in characteristic zero, they are always separable, even if they are the minimal polynomial of an element of the inversive closure. For this reason, we always work in models of what \cite{chatzidakis2004model} call $T_\sigma$. Let $\LK \coloneqq \{+,\cdot,0,1,-,\sigma\}$.

\begin{definition}
    We let $\primary$ be the $\LK$-theory of characteristic zero difference fields $(K,\sigma)$ where $\sigma(K)$ is relatively algebraically closed in $K$. Equivalently, of difference fields $(K,\sigma)$ with $K$ relatively algebraically closed in $\inv{K}$.
\end{definition}

We will often also say that $(K,\sigma)$ \emph{satisfies $\primary$}, or that a valued difference field $(K,\val,\sigma)$ is \emph{a model of $\primary$}, to mean that $(K,\sigma) \vDash \primary$.

\begin{fact}[{\cite[Propositions 4.5, 4.29, 4.30, 4.31]{dor2023contracting}}]\label{separability-cheat-sheet-2}
    Let $(K,\sigma) \subseteq (M,\sigma) \subseteq (L,\sigma)$ be a tower of difference fields. Assume that $(K,\sigma)$ is a model of $\primary$. Then:
    \begin{enumerate}
        \item $(K,\sigma) \subseteq (L,\sigma)$ is $\sigma$-separable if and only it is almost $\sigma$-separable,
        \item if $L = \extnot{K}{\alpha}$ is $\sigma$-algebraic over $K$, then $(K,\sigma) \subseteq (L,\sigma)$ is $\sigma$-separable if and only if there is $p(X) \in \polring{K}$ with $p(\alpha) = 0$ and $p'(\alpha) \neq 0$,
        \item if $(K,\sigma) \subseteq (L,\sigma)$ is $\sigma$-algebraic, then it is $\sigma$-separable if and only if for every $\alpha \in L$ there exists $p(X) \in \polring{K}$ with $p(\alpha) = 0$ and $p'(\alpha) \neq 0$,
        \item if $(K,\sigma) \subseteq (M,\sigma)$ is $\sigma$-algebraic and $(M,\sigma)$ also is a model of $\primary$, then $(K,\sigma) \subseteq (L,\sigma)$ is $\sigma$-separable if and only if both $(K,\sigma) \subseteq (M,\sigma)$ and $(M,\sigma) \subseteq (L,\sigma)$ are.
    \end{enumerate}
\end{fact}
\par When a (valued) difference field doesn't satisfy $\primary$, one can canonically move to a minimal extension that satisfies it, which we call the $\primary$-closure.
\begin{lemma}[\nonsurj]\label{primary-closure}
    Let $(K,\val,\sigma)$ be a valued difference field. Then $(K,\val,\sigma)$ admits an extension $\FE{K}$ that is a model of $\primary$ and with the following universal mapping property: for any extension $(L,\val,\sigma)$ of $(K,\val,\sigma)$ which is a model of $\primary$, there is a unique embedding $f\colon \FE{K} \to L$ over $K$. If such an $L$ exists that is $\sigma$-separable over $K$, then $K = \FE{K}$.
\end{lemma}
\begin{proof}
    We let $\FE{K} \coloneqq K^{\mathrm{alg}} \cap \inv{K}$. Then, $\FE{K} \subseteq \inv{K}$ gives a unique valued difference field structure on $\FE{K}$. Furthermore, since $L$ is relatively algebraically closed in $\inv{L}$, one gets a unique embedding of $\FE{K}$ into $L$ over $K$.
\end{proof}

\begin{remark}
    If $\Res{\sigma}$ and $\Val{\sigma}$ are surjective, then $K \subseteq \FE{K}$ is an immediate extension.
\end{remark}

\begin{definition}
    We call $\FE{K}$ the \emph{$\primary$-closure} of $K$.
\end{definition}

\par When building a back-and-forth system, it will be crucial that when we move to the $\primary$-closure inside a $\sigma$-separable extension, this $\sigma$-separability carries over.

\begin{lemma}
    \label{primary-equivalence}
    Suppose $(K,\sigma) \subseteq (L,\sigma)$ is a difference field extension, and $(K,\sigma)$ is a model of $\primary$. Let $(L,\sigma) \subseteq (L',\sigma)$ be a difference field extension, with $L \subseteq L'$ algebraic. Then $(K,\sigma) \subseteq (L,\sigma)$ is $\sigma$-separable if and only if $(K,\sigma) \subseteq (L',\sigma)$ is $\sigma$-separable. In particular, $(K,\sigma) \subseteq (L,\sigma)$ is $\sigma$-separable if and only if $(K,\sigma) \subseteq (\FE{L},\sigma)$ is $\sigma$-separable.
\end{lemma}
\begin{proof}
    As $K$ is a model of $\primary$, $(K,\sigma) \subseteq (L,\sigma)$ is $\sigma$-separable if and only if it is almost $\sigma$-separable, by \cref{separability-cheat-sheet-2}(1). The statement then follows from the fact that $L \subseteq L'$ is an algebraic extension.
\end{proof}

\begin{lemma}\label{inversive-relative-transc}
    Let $(K,\sigma) \subseteq (L,\sigma)$ be an extension of difference fields, both models of $\primary$. Let $\tau \in \core{L}$ be $\sigma$-transcendental over $K$. Then $\ext{K}{\tau} \subseteq L$ is a $\sigma$-separable extension. In particular, $\ext{K}{\tau}$ is a model of $\primary$.
\end{lemma}
\begin{proof}
    Let $K_0 \coloneqq \ext{K}{\tau}$. Note that $K \subseteq K_0$ is a $\sigma$-separable extension.
    Then, by part (3) of \cref{separability-cheat-sheet}, $K_0 \subseteq L$ is $\sigma$-separable if and only if $K_0 \otimes_K \inv{K} \subseteq L \otimes_K \inv{K}$ is $\sigma$-separable. But this is trivially satisfied, since $K_0 \otimes_K \inv{K}$ is inversive.
\end{proof}

\runin{The relative $\sigma$-separably $\sigma$-algebraic closure} Let $(K,\sigma) \subseteq (L,\sigma)$. We say that $(K,\sigma)$ is \emph{relatively $\sigma$-separably $\sigma$-algebraically closed} in $(L,\sigma)$ if there is no proper intermediate $K \subset M \subseteq L$ with $(K,\sigma) \subset (M,\sigma)$ $\sigma$-separably $\sigma$-algebraic. If $(K,\sigma)$ is a model of $\primary$, this is equivalent to: for all $p(X) \in \polring{K}$ with $p' \neq 0$, if $b \in L$ is such that $p(b) = 0$, then $b \in K$. Note that, in contrast with the case for fields, this is strictly stronger than simply requiring that whenever $p(X) \in \polring{K}$ with $p' \neq 0$ has a solution in $L$, it also has a solution in $K$.

\begin{example}
    $(K,\sigma)$ is relatively $\sigma$-separably $\sigma$-algebraically closed in $(\inv{K},\sigma)$.
\end{example}

\begin{fact}[\nonsurj, {\cite[Theorem 4.46(1)-(2)]{dor2023contracting}}]\label{sep-alg-closure}
    Let $(K,\val,\sigma)$ be a valued difference field which is a model of $\primary$, with $\Res{\sigma}$ and $\Val{\sigma}$ not necessarily surjective. Let $(L,\val,\sigma)$ be an extension which is a model of $\primary$. Let $\widetilde{K} \coloneqq \{a \in L \mid \exists f(X) \in \polring{K} (f(a) = 0 \land f'(a) \neq 0)\}$.
    Then,
    \begin{enumerate}
        \item $\widetilde{K}$ is relatively $\sigma$-separably $\sigma$-algebraically closed in $L$,
        \item $\widetilde{K}$ is $\sigma$-separably $\sigma$-algebraic over $K$ and satisfies $\primary$,
        \item if $(K,\val,\sigma) \subseteq (K',\val,\sigma)$ is a $\sigma$-separably $\sigma$-algebraic extension, then every embedding of $K'$ in $L$ over $K$ has image in $\widetilde{K}$.
    \end{enumerate}
\end{fact}

\par We call $\widetilde{K}$ as in \cref{sep-alg-closure} the \emph{relative $\sigma$-separably $\sigma$-algebraic closure} of $K$ in $L$.

\subsection{$\sigma$-separable generation} Drawing parallels with MacLane's structure theorem of finitely generated separable extensions, we recall the crucial results in \cite{dor2023contracting} about finitely generated $\sigma$-separable extensions.

\begin{definition}
    Let $(K,\sigma) \subseteq (L,\sigma)$ be an extension of difference fields. We say that $\mathcal{X} \subseteq L$ is a \emph{$\sigma$-transcendence basis} of $L$ over $K$ if $\mathcal{X}$ is $\sigma$-algebraically independent over $K$ and $\extnot{K}{\mathcal X} \subseteq L$ is $\sigma$-algebraic.
\end{definition}

\begin{definition}[{\cite[Definition~4.40]{dor2023contracting}}]
    Let $(K,\sigma) \subseteq (L,\sigma)$ be an extension of models of $\primary$. We say that $L$ is \emph{$\sigma$-separably generated} over $K$ if there is a $\sigma$-transcendence basis $\mathcal X \subseteq L$ of $L$ over $K$ such that $\extnot{K}{\mathcal X} \subseteq L$ is $\sigma$-separably $\sigma$-algebraic. If such $\mathcal X$ can be chosen finite, then we say that $L$ is \emph{finitely $\sigma$-separably generated} over $K$.
\end{definition}

\begin{remark}\label{one-of-them}
    Suppose $L$ is $\sigma$-separably generated over $K$ (thus in particular the extension is $\sigma$-separable), and let $\mathcal X$ be a $\sigma$-transcendence basis witnessing that. Let $x \in \mathcal X$: then, $\extnot{K}{x} \subseteq L$ is $\sigma$-separable. Since the composition of $\sigma$-separable extensions is $\sigma$-separable, it is enough to check that $\extnot{K}{x} \subseteq \extnot{K}{\mathcal X}$ is $\sigma$-separable. Since $x$ is transcendental over $K \otimes_{\sigma(K)} \sigma(\extnot{K}{\mathcal X})$, this is true by \cite[Proposition 4.50(2)]{dor2023contracting}.
\end{remark}

\begin{fact}[{\cite[Theorem 4.46(5)]{dor2023contracting}}]\label{maclane}
    Let $(K,\sigma) \subseteq (L,\sigma)$ be a $\sigma$-separable extension of models of $\primary$. Suppose that $L$ is finitely generated as a model of $\primary$ over $K$, meaning that there is a finite $\mathcal X \subseteq L$ such that $L = \FE{\extnot{K}{\mathcal X}}$. Then, $L$ is finitely $\sigma$-separably generated over $K$.
\end{fact}

%% file: henselianity-and-kaplansky.tex
This section constitutes the bulk of the tools that we will need for our relative quantifier elimination. The crucial ingredient for an embedding lemma is some notion of henselianity, i.e. a recipe for producing solutions to (difference) polynomial equations. We adapt the notion defined in \cite{onay2013quantifier}: since we deal with a possibly non-surjective $\sigma$, we have to rule out the possibility of solving the equation $\sigma(X) = a$, for $a \in K$. This is achieved by the definition of \textit{weakly} $\sigma$-henselian.

\runin{$\sigma$-henselianity} In the pure valued field world, henselianity can be phrased in many ways. Crucially it can be seen as a first-order shadow of completeness, meaning that approximate roots of polynomials give rise to proper roots, granted that they are residually simple. In the valued difference world, since we allow all kinds of behaviour of $\Val{\sigma}$, it is impossible to predict the behaviour of the values of $a, \sigma(a), \cdots \sigma^n(a)$, as $\val(a)$ grows larger; it is thus not enough to check that certain derivatives do not vanish residually, but one has to consider all of them at the same time. This leads to the definition introduced in \cite{onay2013quantifier} which is, at face value, more artificial and harder to parse than the one introduced in \cite{belair2007model} or \cite{dor2023contracting} (where the behaviour of $\Val{\sigma}$ is fixed).

\begin{definition}
    Let $p \in \polring{K}$ be a non-constant difference polynomial, and $a \in K$. We say that $p$ is in \emph{$\sigma$-henselian configuration} at $a$ if there are $\gamma \in \valuegroup{K}$ and $0 \leq i \leq n$ such that 
    \begin{enumerate}
        \item $\val(p(a)) = \val(p_{i}(a)) + \Val{\sigma}^i(\gamma) \leq \val(p_{j}(a)) + \Val{\sigma}^j(\gamma)$ for all $0 \leq j \leq n$,
        \item $\val(p_{J}(a)) + J(\gamma) < \val(p_{J+L}(a)) + (J+L)(\gamma)$ whenever $J, L \neq 0$ and $p_{J} \neq 0$.
    \end{enumerate}
\end{definition}
The $\gamma$ as above is unique, and we denote it by $\gamma(p,a)$.
\begin{remark}
    An instructive case to consider is the one where $p(X)$ is a non-constant difference polynomial, and $a$ is such that $p(a) \neq 0$, $\val(p(a)) > 0$, and for every $I$ such that $p_{I} \neq 0$, $\val(p_{I}(a)) = 0$. Then, $p$ is in $\sigma$-henselian configuration at $a$.
\end{remark}

\par The following definition is what \cite{onay2013quantifier} call the $\sigma$-Hensel scheme.
\begin{definition}
    \label{def:shens}
    We will say that $(K,\val,\sigma)$ is \emph{strongly $\sigma$-henselian} if whenever $p$ is in $\sigma$-henselian configuration at $a$, then there is $b \in K$ with $p(b) = 0$ and $\val(b-a) = \gamma(p,a)$.
\end{definition}

\par As the name suggests, strong $\sigma$-henselianity is too strong. In particular, one risks solving equations of the form $\sigma(X) = a$, implying that $\sigma$ is surjective. We thus weaken it for our purposes.

\begin{definition}
    \label{def:whens}
    We will say that $(K,\val,\sigma)$ is \emph{weakly $\sigma$-henselian} if whenever $p$ is in $\sigma$-henselian configuration at $a$, and $p' \neq 0$, then there is $b \in K$ with $p(b) = 0$ and $\val(b-a) = \gamma(p,a)$.
\end{definition}

\par We now explore for a moment the relationship between strong and weak $\sigma$-henselianity. By direct computation, one can prove the following lemma.

\begin{lemma}\label{the-horror}
    Let $n \in \mathbb{Z}$, and assume that $p$ is in $\sigma$-henselian configuration at $a$. Then, $p^{\sigma^n}$ is in $\sigma$-henselian configuration at $\sigma^n(a)$. Moreover, if $m$ is the smallest such that $p_m \neq 0$ (see \cref{shift-reference}), then $\shifted{p}^{\sigma^{-m}}$ is in $\sigma$-henselian configuration at $a$.
\end{lemma}

\par We can now use this to move between a valued difference field and its inversive closure.

\begin{lemma}\label{surj-strong}
    Suppose $(K,\val,\sigma)$ is weakly $\sigma$-henselian and $\sigma$ is surjective. Then it is strongly $\sigma$-henselian.
\end{lemma}
\begin{proof}
    Let $p$ be in $\sigma$-henselian configuration at $a$. If $m = m(p)$, then let $q = \shifted{p}^{\sigma^{-m}}$: by \cref{the-horror}, $q$ is in $\sigma$-henselian configuration at $a$, and further $q' \neq 0$. As $\sigma$ is surjective, $q$ is still a difference polynomial over $K$, so by weak $\sigma$-henselianity we find $b \in K$ such that $q(b) = p^{\sigma^{-m}}(\sigma^{-m}(b)) = \sigma^{-m}(p(b)) = 0$, and thus $p(b) = 0$. Moreover, $\val(b-a) = \gamma(q,a) = \gamma(p,a)$.
\end{proof}

\begin{lemma}\label{inv-strongly-hens}
    Let $(K,\val,\sigma)$ be a valued difference field that is a model of $\primary$. Then, $(K,\val,\sigma)$ is weakly $\sigma$-henselian if and only if $\inv{K}$ is strongly $\sigma$-henselian.
\end{lemma}
\begin{proof}
    Assume that $(K,\val,\sigma)$ is weakly $\sigma$-henselian. Note that since being weakly $\sigma$-henselian forms an $\forall\exists$-theory, and $\inv{K}$ is the union of copies of $K$ (and hence models of this theory), $\inv{K}$ is also weakly $\sigma$-henselian. Since it is surjective, \cref{surj-strong} implies that $\inv{K}$ is strongly $\sigma$-henselian.
    \par Viceversa, assume that $\inv{K}$ is strongly $\sigma$-henselian. If $p \in \polring{K}$ is in $\sigma$-henselian configuration at $a \in K$, and $p' \neq 0$, then we can find $b \in \inv{K}$ such that $p(b) = 0$ and $\val(b-a) = \gamma(p,a)$. However, $K \subseteq \inv{K}$ is relatively $\sigma$-separably $\sigma$-algebraically closed, thus $b \in K$ already.
\end{proof}

\par The authors of \cite{onay2013quantifier} consider valued difference fields which are (strongly) $\sigma$-henselian and furthermore have linearly difference closed residue difference fields. However, as observed in \cite{azgin2010valued}, this already follows from (strong) $\sigma$-henselianity.

\begin{lemma}\label{hens-lin-closed}
    Let $(K,\val,\sigma)$ be weakly $\sigma$-henselian. Then $(\residue{K},\Res{\sigma})$ is \emph{linearly difference closed}, i.e. if for every $\alpha_0, \dots \alpha_n \in \residue{K}$, at least one of them not zero, there is $z \in \residue{K}$ such that
    \[
        1+\alpha_0 z + \alpha_1 \Res{\sigma}(z) + \dots + \alpha_n \Res{\sigma}^n(z) = 0.
    \]
\end{lemma}
\begin{proof}
    Since $K \subseteq \inv{K}$ is an immediate extension, it is enough to show that $\inv{K}$ has linearly difference closed residue difference field. We let $\alpha_i = \res(a_i)$, for $i = 0, \dots n$, and assume $S = \{j \leq n \mid \alpha_j \neq 0\} \neq \emptyset$. Consider the difference polynomial $p(X) \coloneqq 1+ \sum_{i \in S} a_i X^i$. By construction, $p_i(X) = 0$ if and only if $i \notin S$, i.e. if and only if $\alpha_i = 0$, and $p_{J} = 0$ for any $J$ with $\vert J \vert > 1$. We check that $p$ is in $\sigma$-henselian configuration at $0$: indeed, using $\gamma = 0$,
    \[
        \val(p(0)) = \val(1) = 0 = \val(p_j(0)) = \val(a_j) = 0
    \]
    for all $j \in S$. By \cref{surj-strong}, $\inv{K}$ is strongly $\sigma$-henselian, and thus we find $b \in \inv{K}$ such that $p(b) = 0$ and $\val(b) = \gamma(p,0) = 0$. Then, $\beta = \res(b)$ is the required solution of the equation in $\residue{K}$.
\end{proof}

\runin{$\sigma$-ramification} We establish that a weakly $\sigma$-henselian valued difference field is dense in its inversive closure, something similar to how a separably closed valued field is dense in its perfect hull. Introduced in \cite{dor2023contracting}, we call this phenomenon \textit{deep $\sigma$-ramification}, and deduce that in $\aleph_0$-saturated models, the mirrored situation occurs where the inversive core is dense. Once again, this is reflected in $\aleph_0$-saturated separably closed valued fields, which contain a dense algebraically closed valued field, namely their perfect core.

\begin{definition}\label{deep-transf}
    We say that $(K,\val,\sigma)$ is \emph{deeply $\sigma$-ramified} if $\sigma(K) \subseteq K$ is dense (equivalently, if $K \subseteq \inv{K}$ is dense).
\end{definition}
\begin{lemma}\label{as-deep}
    Suppose $(K,\val,\sigma)$ is weakly $\sigma$-henselian. Then for any $a \in \O_K$ and $\epsilon \in \m_K\setminus\{0\}$, there is $b \in \O_K$ such that $\sigma(b)-\epsilon b - a = 0$.
\end{lemma}
\begin{proof}
    Let $p(X) \coloneqq \sigma(X) - \epsilon X - a$. Note that $p' \neq 0$. Moreover, let $\beta \in \residue{K}$ be such that $\Res{\sigma}(\beta) \neq \res(a)$, and let $b \in \O_K$ be a lift of $\beta$. We argue that $p$ is in $\sigma$-henselian configuration at $b$: take $\gamma = 0$, then
    \begin{enumerate}
        \item $\val(p(b)) = 0 = \val(p_1(b)) = \val(1) < \val(p_0(b)) = \val(\epsilon)$,
        \item for all $J, L \neq 0$, $p_{J+L} = 0$, thus the second part is trivially satisfied.
    \end{enumerate}
    By weak $\sigma$-henselianity, there is $b' \in K$ with $p(b') = 0$ and $\val(b'-b) = 0$. In particular, $\val(b') = \val(b'+b-b) \geq \min\{0, \val(b)\} \geq 0$.
\end{proof}

\begin{corollary}\label{hens-implies-inv}
    Suppose $(K,\val,\sigma)$ is weakly $\sigma$-henselian. Then, $K$ is deeply $\sigma$-ramified. In particular, both $\Res{\sigma}$ and $\Val{\sigma}$ are surjective.
\end{corollary}
\begin{proof}
    Given any $a \in K$ and $\gamma > 0$, we need to find $b \in K$ such that $\sigma(b) \in B_\gamma(a)$.
    \par First, assume that $\val(a) \geq 0$. We let $\epsilon \in K$ be such that $\val(\epsilon) > \max( \val(a),\gamma)$, and we let $b \in \O_K$ be such that $\sigma(b) - \epsilon b - a = 0$. Then, $\val(\sigma(b)-a) = \val(\epsilon b) \geq \val(\epsilon) > \gamma$, so $\sigma(b) \in B_\gamma(a)$, as required.
    \par If $\val(a) < 0$, we rescale $a' = a\sigma(e)$ and $\gamma' = \gamma + \Val{\sigma}( \val(e))$, where $\val(\sigma(e)) \geq -\val(a) > 0$. We are then back to the first case: we let $\val(\epsilon) > \max(\val(a'),\gamma')$, and $b \in \O_K$ be such that $\sigma(b)-\epsilon b - a' = 0$. Then,
    \[
      \val\left(\sigma\left(\frac{b}{e}\right)-a\right) \geq \val(\epsilon) - \Val{\sigma}(\val(e)) > \gamma' - \Val{\sigma}(\val(e)) = \gamma.
    \]
    Now, since $K$ is dense in $\inv{K}$, in particular the extension $K \subseteq \inv{K}$ is immediate.
\end{proof}

\begin{lemma}\label{core-dense}
    Let $(K,\val,\sigma)$ be weakly $\sigma$-henselian and $\aleph_0$-saturated. Then, $\core{K}$ is dense in $K$ and strongly $\sigma$-henselian.
\end{lemma}
\begin{proof}
    By $\aleph_0$-saturation and the fact that $\sigma(K) \subseteq K$ is dense (by \cref{hens-implies-inv}), one gets that $\core{K} \subseteq K$ is dense.
    Furthermore, as $\sigma$ is surjective on $\core{K}$, it is enough by \cref{surj-strong} to show that $\core{K}$ is weakly $\sigma$-henselian. Note that $\sigma^n$ gives an isomorphism $K \cong \sigma^n(K)$, thus each $\sigma^n(K)$ is weakly $\sigma$-henselian. Now, if $p \in \polring{\core{K}}$ is in $\sigma$-henselian configuration at $a \in \core{K}$, then for every $n \geq 0$ we have $p \in \polring{\sigma^n(K)}$ and $a \in \sigma^n(K)$, and thus we can find some $b_n \in \sigma^n(K)$ such that $p(b_n) = 0$ and $\val(b_n-a) = \gamma(p,a)$. By saturation, then, we find $b \in \core{K}$ such that $p(b) = 0$ and $\val(b-a) = \gamma(p,a)$, as required.
\end{proof}

\runin{Immediate extensions}\label{sub:immediate} The next necessary tool to establish relative quantifier elimination is a reasonable theory of immediate extensions, which usually goes under the umbrella of \textit{Kaplansky theory}.

\begin{definition}\label{tsam}
    We say that $(K,\val,\sigma)$ is \emph{($\sigma$-separably) $\sigma$-algebraically maximal} if it has no proper immediate ($\sigma$-separably) $\sigma$-algebraic extension.
\end{definition}

\begin{lemma}
    \label{transf-alg-max}
    Suppose $(K,\val,\sigma)$ is a model of $\primary$. Then, $K$ is $\sigma$-separably $\sigma$-algebraically maximal if and only if $\inv{K}$ is $\sigma$-algebraically maximal.
\end{lemma}
\begin{proof}
    Suppose $K$ is $\sigma$-separably $\sigma$-algebraically maximal, and suppose $\inv{K}$ is not $\sigma$-algebraically maximal. Let $\inv{K} \subseteq \ext{\inv{K}}{\alpha}$ be an immediate $\sigma$-algebraic extension. Let $f(X) \in \polring{\inv{K}}$ be a difference polynomial such that $f(\alpha) = 0$, and write
    \[
        f(X) = \sum_{I} \sigma^{-n}(b_{I}) X^{I},
    \]
    for some $n \geq 0$, and $b_{I} \in K$. If $g = \shifted{f^{\sigma^n}}$, then $g$ is defined over $K$, $g' \neq 0$, and $g(\sigma^m(\alpha)) = 0$, where $m$ is the smallest such that $f^{\sigma^n}_m \neq 0$ (see \cref{shift-reference}). As $\alpha \notin \inv{K}$, then $\sigma^m(\alpha) \notin K$, and thus $K \subseteq \extnot{K}{\sigma^m(\alpha)}$ is a proper extension. Moreover, it is a subextension of the immediate extension $K \subseteq \ext{\inv{K}}{\alpha}$, so it is a proper, immediate and, by \cref{separability-cheat-sheet-2}(2), $\sigma$-separably $\sigma$-algebraic extension of $K$. This is a contradiction.
    
    Viceversa, suppose $\inv{K}$ is $\sigma$-algebraically maximal, but $K$ is not $\sigma$-separably $\sigma$-algebraically maximal. Let $K \subseteq \extnot{K}{\alpha}$ be a $\sigma$-separably $\sigma$-algebraic immediate extension, thus in particular $\alpha \notin \inv{K}$. Then, $K \subseteq \extnot{K}{\alpha} \subseteq \inv{\extnot{K}{\alpha}}$ is a tower of immediate extensions which contains $\inv{K}$, thus in particular the subextension $\inv{K} \subseteq \extnot{\inv{K}}{\alpha}$ is still immediate. As $\alpha \notin \inv{K}$, this is a proper immediate $\sigma$-algebraic extension, a contradiction.
\end{proof}

\begin{lemma}\label{existence-diamond}
    Suppose $(K,\val,\sigma)$ is a model of $\primary$. Then it admits a $\sigma$-separably $\sigma$-algebraically maximal, $\sigma$-separably $\sigma$-algebraic immediate extension $K'$ that is a model of $\primary$.
\end{lemma}
\begin{proof}
    Using Zorn's lemma, we let $K'$ be maximal among the $\sigma$-separably $\sigma$-algebraic immediate extensions of $K$ that be a model of $\primary$. We argue that $K'$ is $\sigma$-separably $\sigma$-algebraically maximal. Suppose not, i.e. there is some proper $\sigma$-separably $\sigma$-algebraic immediate extension $K' \subseteq L$. Then, by \cref{primary-equivalence}, $K' \subseteq \FE{L}$ is still a proper $\sigma$-separably $\sigma$-algebraic immediate extension, now that is a model of $\primary$, which is a contradiction.
\end{proof}

\begin{fact}[{\cite[Theorem 5.8]{onay2013quantifier}}]
    \label{inversive-kaplansky}
    Suppose $(K,\val,\sigma)$ is inversive, and $(\residue{K},\Res{\sigma})$ is linearly difference closed. Then all its $\sigma$-algebraically maximal immediate $\sigma$-algebraic extensions are isomorphic over $K$.
\end{fact}

\par We are now ready to prove the main ingredient of \cref{embedding-theorem}, namely the uniqueness of certain maximal immediate extensions. We first prove an asymmetric version of the uniqueness theorem.

\begin{proposition}\label{partial-kaplansky}
    Suppose $(K,\val,\sigma)$ is a model of $\primary$ and $(\residue{K},\Res{\sigma})$ is linearly difference closed. Let $(K_1,\val,\sigma)$ be $\sigma$-separably $\sigma$-algebraically maximal, $\sigma$-separably $\sigma$-algebraic over $K$, and a model of $\primary$. Let $(K_2,\val,\sigma)$ be $\sigma$-separably $\sigma$-algebraically maximal, $\sigma$-separable and immediate over $K$, and a model of $\primary$. Then there is an embedding $\phi\colon K_1 \hookrightarrow K_2$ over $K$ such that $\phi(K_1) \subseteq K_2$ is $\sigma$-separable.
\end{proposition}
\begin{proof}
    
    By \cref{transf-alg-max}, $\inv{K_1}$ and $\inv{K_2}$ are both $\sigma$-algebraically maximal, and further both are $\sigma$-algebraic immediate extensions of $K$, and thus of $\inv{K}$. Hence, by \cref{inversive-kaplansky}, there is a $\inv{K}$-isomorphism $\phi \colon \inv{K_1} \to \inv{K_2}$. We use $\phi$ to embed $K_1$ into $K_2$ over $K$.
    \par Given any $\alpha \in K_1$, we argue that $\phi(\alpha) \in K_2$. Suppose not, i.e. $\phi(\alpha) \in \inv{K_2} \setminus K_2$; equivalently, $K_2 \subseteq \extnot{K_2}{\phi(\alpha)}$ is a purely $\sigma$-inseparable extension. Since $K_1$ is $\sigma$-separably $\sigma$-algebraic over $K$, we let $f(X) \in \polring{K}$ be a difference polynomial with $f(\alpha) = 0$ and $f'(\alpha) \neq 0$. Then, $\phi(\alpha)$ is also such that $f(\phi(\alpha)) = 0$ and $f'(\phi(\alpha)) \neq 0$. If we see $f(X)$ as a difference polynomial over $K_2$, we get that $\extnot{K_2}{\phi(\alpha)}$ is a $\sigma$-separably $\sigma$-algebraic immediate extension of $K_2$. This is a contradiction.
    \par Thus, $\phi(K_1) \subseteq K_2$ and, since $K \subseteq \phi(K_1)$ is a $\sigma$-separably $\sigma$-algebraic extension, then $\phi(K_1) \subseteq K_2$ is $\sigma$-separable, but not necessarily $\sigma$-algebraic.
\end{proof}

\begin{theorem}\label{separable-kaplansky}
    Suppose $(K,\val,\sigma)$ is a model of $\primary$ and $(\residue{K},\Res{\sigma})$ is linearly difference closed.    
    Let $(K_1,\val,\sigma)$ and $(K_2,\val,\sigma)$ be two $\sigma$-separably $\sigma$-algebraically maximal, $\sigma$-separably $\sigma$-algebraic immediate extensions of $K$ that are also models of $\primary$.
    Then $K_1 \cong_K K_2$.
\end{theorem}
\begin{proof}
    From \cref{partial-kaplansky}, we get an embedding $\phi \colon K_1 \to K_2$ over $K$ such that $\phi(K_1) \subseteq K_2$ is $\sigma$-separable. Then, since $K_2$ is $\sigma$-algebraic over $K$, $\phi(K_1) \subseteq K_2$ is $\sigma$-separably $\sigma$-algebraic and immediate. Thus $\phi$ must be surjective.
\end{proof}

\begin{definition}\label{unique-maximal-def}
    Whenever $K$ is as above, we denote the unique $\sigma$-separably $\sigma$-algebraically maximal, $\sigma$-separably $\sigma$-algebraic immediate extension that is a model of $\primary$ of $K$ by $\smax{K}$.
\end{definition}
\begin{remark}
    Note that, by \cref{transf-alg-max}, $\invmax{K}$ is $\sigma$-algebraically maximal.
\end{remark}

\par Before continuing, we establish that (as one might expect) a $\sigma$-separably $\sigma$-algebraically maximal valued difference field is in fact weakly $\sigma$-henselian.

\begin{fact}[{\cite[Corollary 5.6(2)]{onay2013quantifier}}]\label{maximality-hens}
    Let $(K,\val,\sigma)$ be $\sigma$-algebraically maximal, and $(\residue{K},\Res{\sigma})$ is linearly difference closed. Then $(K,\val,\sigma)$ is strongly $\sigma$-henselian.
\end{fact}
\begin{lemma}\label{tsam-is-hens}
    Let $(K,\val,\sigma)$ be a model of $\primary$, and suppose $(\residue{K},\Res{\sigma})$ is linearly difference closed. Then, $\smax{K}$ is weakly $\sigma$-henselian.
\end{lemma}
\begin{proof}
    By \cref{maximality-hens}, $\invmax{K}$ is strongly $\sigma$-henselian. As $\smax{K}$ is relatively $\sigma$-separably $\sigma$-algebraically closed in $\invmax{K}$ and a model of $\primary$, $\smax{K}$ is weakly $\sigma$-henselian.
\end{proof}

\begin{corollary}\label{max-deep-ram}
    Let $(K,\val,\sigma)$ be a model of $\primary$, and suppose $(\residue{K},\Res{\sigma})$ is linearly difference closed. Then $\smax{K}$ is deeply $\sigma$-ramified.
\end{corollary}
\begin{proof}
    This is a consequence of \cref{tsam-is-hens} and \cref{hens-implies-inv}.
\end{proof}

\runin{Pseudo-Cauchy sequences and pseudolimits}
So far, we have avoided talking about pseudo-Cauchy sequences, sweeping them under the rug of \cref{inversive-kaplansky}. We now ought to make a stop and discuss them, for the sole purpose of establishing \cref{dependent-defect}.

Given a limit ordinal $\lambda$, we call $(a_\rho)_{\rho<\lambda} \subseteq K$ a \emph{pseudo-Cauchy sequence} if there is $\overline{\rho} < \lambda$ such that, for every $\rho_2 > \rho_1 > \rho_0 \geq \overline{\rho}$,
\[
  \val(a_{\rho_2}-a_{\rho_1}) > \val(a_{\rho_1}-a_{\rho_0}).
\]
For any $\rho \geq \overline{\rho}$, we write $\gamma_\rho \coloneqq \val(a_{\rho+1}-a_{\rho})$; note that $\gamma_\rho = \val(a_{\mu}-a_{\rho})$ for every $\mu > \rho$. We call the sequence $(\gamma_\rho)_{\rho<\lambda}$ the \emph{radii} of the sequence. If $(\gamma_\rho)_{\rho<\lambda} \subseteq \valuegroup{K}$ is cofinal, we also say that $(a_\rho)_{\rho<\lambda}$ is a \emph{Cauchy sequence}. We say that $a \in L$, for some extension $L$ of $K$, is a \emph{pseudolimit} of $(a_\rho)_\rho$ if $(\val(a-a_\rho))_\rho$ is eventually strictly increasing (equivalently, $\val(a-a_\rho) = \gamma_\rho$ for $\rho$ big enough). We then write $a_\rho \pseudoconverges a$. Given another $b \in L$, we have that $a_\rho \pseudoconverges b$ if and only if $\val(b-a) \geq \gamma_\rho$ eventually in $\rho$. We say that two sequences $(a_\rho)_\rho$ and $(b_\rho)_\rho$ are \emph{equivalent}, and we write $(a_\rho)_\rho \sim (b_\rho)_\rho$, if for every extension $L$ and $a \in L$, $a_\rho \pseudoconverges a$ if and only if $b_\rho \pseudoconverges a$.

In general, one doesn't have that, for any difference polynomial $p(X)$, if $a_\rho \pseudoconverges a$, then $p(a_\rho) \pseudoconverges p(a)$.\footnote{This is true if $\Val{\sigma}$ is $\omega$-increasing, i.e. for all $n > 0$ and $\gamma > 0$, $\Val{\sigma}(\gamma) > n\gamma$.} This can be easily fixed, however, by moving to an equivalent pseudo-Cauchy sequence.

\begin{definition}
    We say that a difference field $(K,\sigma)$ is \emph{aperiodic} if for every $n > 0$ there is $\alpha \in K$ with $\sigma^n(\alpha) \neq \alpha$.
\end{definition}

\begin{remark}
    If $(K,\sigma)$ is linearly difference closed, then it is in particular aperiodic.
\end{remark}

\begin{fact}[{\cite[Theorem 3.8]{onay2013quantifier}}]
    Suppose $(K,\val,\sigma)$ is a valued difference field and $(\residue{K},\Res{\sigma})$ is aperiodic. Suppose $(a_\rho)_\rho \subseteq K$ is a pseudo-Cauchy sequence. Take $a$ in some extension of $K$ with $a_\rho \pseudoconverges a$. Let $\Sigma \subseteq \polring{K}$ be finite. Then, there is a pseudo-Cauchy sequence $(b_\rho)_\rho \subseteq K$ with $(a_\rho)_\rho \sim (b_\rho)_\rho$ and such that for every non-constant $p \in \Sigma$, $p(b_\rho) \pseudoconverges p(a)$.
\end{fact}

Given a difference polynomial $p(X)$, in particular, we have that $p(a_\rho) \pseudoconverges 0$ if and only if $(\val(p(a_\rho)))_\rho$ is eventually strictly increasing.

\begin{definition}
    Let $(a_\rho)_\rho \subseteq K$ be a pseudo-Cauchy sequence. For $I \in \mathbb{N}^3$, denote by $\polring{K}^I$ the set of difference polynomials of complexity $I$. Let
    \[
        \Wit_I((a_\rho)_\rho) = \{p(X) \in \polring{K}^I \mid \exists (b_\rho)_\rho \sim (a_\rho)_\rho(p(b_\rho) \pseudoconverges 0)\},
    \]
    and $\Wit((a_\rho)_\rho) = \bigcup_{I \in \mathbb N^3} \Wit_I((a_\rho)_\rho)$. We say that $(a_\rho)_\rho$ is:
    \begin{enumerate}
        \item of \emph{$\sigma$-transcendental type} if $\Wit((a_\rho)_\rho) = \varnothing$,
        \item of \emph{$\sigma$-algebraic type} otherwise.
    \end{enumerate}
    If $(a_\rho)_\rho$ is of $\sigma$-algebraic type, and $J$ is the smallest such that $\Wit_{J}((a_\rho)_\rho)$ is non-empty, then we write $\Wit_{\mathrm{min}}((a_\rho)_\rho)$ for $\Wit_J((a_\rho)_\rho)$.
\end{definition}
\begin{definition}
    We say that $(a_\rho)_\rho$ is of \emph{$\sigma$-separably $\sigma$-algebraic type} if it is of $\sigma$-algebraic type and there is $f \in \Wit_{\mathrm{min}}((a_\rho)_\rho)$ with $f' \neq 0$.
\end{definition}

\par The following is a result which is well-known, and can be seen by direct computation; for example it is partially shown in \cite[Lemma 6.2]{azgin2010valued} and \cite[Lemma 7.2]{belair2007model}.
\begin{lemma}[\nonsurj]\label{compute-eventual-value}
    Let $(K,\val,\sigma)$ be a valued difference field, with $(\residue{K},\Res{\sigma})$ aperiodic, and let $(a_\rho)_\rho \subseteq K$ be a pseudo-Cauchy sequence without pseudolimit in $K$. Let $I$ be either the minimal complexity of a difference polynomial pseudoconverging to $0$ on some $(b_\rho)_\rho \sim (a_\rho)_\rho$, if there are any, or $\infty > \mathbb{N}^3$ otherwise. Let $f(X) \in \polring{K}$ have complexity strictly smaller than $I$, and $(L,\val,\sigma)$ be some extension of $K$ containing a pseudolimit $a$ of $(a_\rho)_\rho$. Then,
    \begin{enumerate}
        \item $\val(f(a)) \in \valuegroup{K}$, and
        \item if $\val(f(a)) \geq 0$, then $\res(f(a)) \in \residue{K}$.
    \end{enumerate}
\end{lemma}

\par We next show how to deal with adjoining pseudolimits.

\begin{lemma}[\nonsurj]\label{adjoin-limit}
    Suppose $(K,\val,\sigma)$ is a valued difference field, with $(\residue{K},\Res{\sigma})$ aperiodic, and let $(a_\rho)_\rho \subseteq K$ be a pseudo-Cauchy sequence without pseudolimit in $K$. Suppose $(L,\val,\sigma)$ is an extension of $K$ and there is $b \in L$ with $a_\rho \pseudoconverges b$. Then,
    \begin{enumerate}
        \item if $(a_\rho)_\rho$ is of $\sigma$-transcendental type, there is an immediate $\sigma$-transcendental extension ${K \subseteq \extnot{K}{a}}$ with $a_\rho \pseudoconverges a$, and further $a \mapsto b$ gives rise to a unique embedding of $\extnot{K}{a}$ into $L$ over $K$,
        \item if $(a_\rho)_\rho$ is of $\sigma$-algebraic type, and $p \in \Wit_{\mathrm{min}}((a_\rho)_\rho)$, then there is an immediate $\sigma$-algebraic extension $K \subseteq \extnot{K}{a}$ with $a_\rho \pseudoconverges a$ and $p(a)= 0$, and further if $p(b) = 0$, then $a \mapsto b$ gives rise to a unique embedding of $\extnot{K}{a}$ into $L$ over $K$. Moreover, if $(a_\rho)_\rho$ is of $\sigma$-separably $\sigma$-algebraic type and $p \in \Wit_{\mathrm{min}}((a_\rho)_\rho)$ is such that $p' \neq 0$, then\footnote{Since $p'$ has complexity strictly lower than that of $p$.} $p'(a) \neq 0$.
    \end{enumerate}
\end{lemma}
\begin{proof}
    The proof of {\cite[Lemmata 2.5 and 2.6]{azgin2011elementary}} translates verbatim.
\end{proof}
\begin{remark}
    If we had started with $(K,\val,\sigma)$ that is a model of $\primary$, then the resulting immediate extension $\extnot{K}{a}$ with $p(a) = 0$, $p'(a) \neq 0$ would be $\sigma$-separably $\sigma$-algebraic.
\end{remark}
\begin{corollary}
    Let $(K,\val,\sigma)$ be $\sigma$-separably $\sigma$-algebraically maximal and that is a model of $\primary$. Assume that $(\residue{K},\Res{\sigma})$ is aperiodic. Then, all pseudo-Cauchy sequences of $\sigma$-separably $\sigma$-algebraic type have a pseudolimit in $K$.
\end{corollary}

\par The following result might seem underwhelming, but it is of fundamental importance in the final steps of the back-and-forth for \cref{embedding-theorem}. In spirit, it comes from the theory of \textit{dependent defect}, as developed in \cite{kuhlmann2023valuation}.

\begin{remark}
    If $K$ is deeply $\sigma$-ramified, i.e. $K \subseteq \inv{K}$ is dense, then equivalently $\inv{K} \subseteq \completion{K}$, where by $\completion{K}$ we denote the completion of $K$ as a valued field. In particular, if $a \in \inv{K}$, then there is a Cauchy sequence from $K$ converging to $a$, and thus $\val(a-K) = \{\val(a-b) \mid b \in K\} \subseteq \valuegroup{K}$ is cofinal.
\end{remark}

\begin{lemma}
    \label{dependent-defect}
    Suppose $(K,\val,\sigma)$ is $\sigma$-separably $\sigma$-algebraically maximal, is a model of $\primary$, and $(\residue{K},\Res{\sigma})$ is linearly difference closed. Let $K \subseteq \extnot{K}{t}$ be an immediate $\sigma$-transcendental extension, and let $(a_\rho)_\rho \subseteq K$ be a pseudo-Cauchy sequence with $t$ as pseudolimit. Then $(a_\rho)_\rho$ is of $\sigma$-transcendental type.
\end{lemma}
\begin{proof}
    Suppose not, and let $p(X) \in \Wit_{\mathrm{min}}((a_\rho)_\rho)$. By \cref{adjoin-limit}, we can find a proper immediate $\sigma$-algebraic extension $K \subseteq \extnot{K}{a}$, with $a_\rho \pseudoconverges a$. For $n$ big enough, we can split the extension into the tower $K \subseteq \extnot{K}{\sigma^n(a)} \subseteq \extnot{K}{a}$, where the bottom part is $\sigma$-separably $\sigma$-algebraic, and hence trivial. We are then left with a purely $\sigma$-inseparable extension $K \subseteq \extnot{K}{a}$. Now, since $K$ is deeply $\sigma$-ramified (\cref{max-deep-ram}), $\val(a-K)$ is cofinal in $\valuegroup{K}$. On the other hand, since $K \subseteq \extnot{K}{t,a} \subseteq \inv{\extnot{K}{t}}$ is a tower of immediate extensions, we can compute $\val(a-t) \in \valuegroup{K}$. Since both are pseudolimits, $\val(a-t) \geq \val(a-K)$; as they are not isomorphic over $K$, because $a$ is $\sigma$-algebraic and $t$ is $\sigma$-transcendental, then $\val(a-t) < \infty$. In particular, $\val(a-K) \leq \val(a-t)$ is not cofinal in $\valuegroup{K}$, a contradiction.
\end{proof}

%% file: auxiliary-steps.tex
\par We now have almost all the tools to establish the embedding lemma and deduce relative quantifier elimination. We introduce the leading terms structure, the languages and theories, and take a brief detour to prove that one can embed $\smax{K}$ in saturated models. We then establish recipes for the auxiliary steps, namely increasing the residue difference field and value difference group.

\runin{The leading terms structure}
Given a valued field $(K,\val)$, we start to define the leading terms structure of $(K,\val)$ by considering the set $\RV_K \coloneqq (K^\times/(1+\m_K)) \cup \{0\}$. We write $\rv$ for the quotient map $K^\times \to \RV_K$, extended to $K$ via $\rv(0) \coloneqq 0$.

We have, for $a, b \in K^\times$,
\[
    \rv(a) = \rv(b) \iff \val(a-b) > \val(a) \iff \val(a-b) > \val(b).
\]

In particular, $\rv(a) = \rv(b)$ implies that $\val(a) = \val(b)$, and so $\val\colon K^\times \to \valuegroup{K}$ induces a map $\val_{\rv}\colon \RV_K^\times \to \valuegroup{K}$. On the other hand, it gives rise to a short exact sequence of groups
\[
    1 \longrightarrow \residue{K}^\times \stackrel{\iota}{\longhookrightarrow} \RV_K^\times \stackrel{\val_{\rv}}{\longtwoheadrightarrow} \valuegroup{K} \longrightarrow 0.
\]

We endow $\RV_K^\times$ with the multiplicative structure inherited from $K^\times$, extended to $\RV_K$ by $a \cdot 0 = 0$ for all $a \in \RV_K$. Further, we define the ternary relation
\[
    \oplus(\alpha,\beta,\gamma) \iff \exists a,b \in K^\times (\rv(a) = \alpha \land \rv(b) = \beta \land \rv(a+b) = \gamma).
\]
We call $(\RV_K,\oplus,\cdot,0,1)$ the \emph{leading terms structure} of $(K,\val)$.
We also write $\alpha \oplus \beta \coloneqq \{\gamma \in \RV_K \mid \oplus(\alpha,\beta,\gamma)\}$. We say that $\alpha \oplus \beta$ is \emph{well-defined} if $\alpha \oplus \beta = \{\gamma\}$ for some $\gamma$, in which case we write also $\alpha \oplus \beta = \gamma$.

\par For $\alpha_1, \dots \alpha_n \in \RV_K$, we write 
\[
    \oplus(\alpha_1, \dots \alpha_n, \beta) \iff \exists a_1, \dots a_n \left(\bigwedge_{i=1}^n \rv(a_i) = \alpha_i \land \rv(a_1+\dots + a_n) = \beta\right).
\]
We again say that $\alpha_1 \oplus \cdots \oplus \alpha_n$ is \emph{well-defined} if there is exactly one $\beta$ such that $\oplus(\alpha_1, \dots \alpha_n, \beta)$.
\begin{lemma}
    Let $\alpha_1, \dots \alpha_n \in \RV_K^\times$. Choose some representatives $a_1, \dots a_n \in K$ so that $\alpha_i = \rv(a_i)$, for $i=1, \dots n$. Then $\alpha_1 \oplus \cdots \oplus \alpha_n$ is well-defined if and only if $\val(a_1 + \dots + a_n) = \min_{i=1, \dots n} \val(a_i)$.
\end{lemma}
Given $\sigma \in \operatorname{End}(K,\val)$, then there is an induced $\Rv{\sigma} \in \operatorname{End}(\RV_K)$. 

\begin{remark}
    By applying the Short Five Lemma (\cite[Exercise 1.3.3]{weibel1994introduction}) to
    \[\begin{tikzcd}
        1 & {\residue{K}^\times} & {\RV_K^\times} & {\valuegroup{K}} & 0 \\
        1 & {\residue{K}^\times} & {\RV_K^\times} & {\valuegroup{K}} & 0
        \arrow[from=1-1, to=1-2]
        \arrow[from=1-2, to=1-3]
        \arrow[from=1-3, to=1-4]
        \arrow[from=1-4, to=1-5]
        \arrow[from=2-1, to=2-2]
        \arrow[from=2-2, to=2-3]
        \arrow[from=2-3, to=2-4]
        \arrow[from=2-4, to=2-5]
        \arrow["\Res{\sigma}", from=1-2, to=2-2]
        \arrow["\Rv{\sigma}", from=1-3, to=2-3]
        \arrow["\Val{\sigma}", from=1-4, to=2-4]
    \end{tikzcd}\]
    we get that $\Rv{\sigma}$ is surjective if and only if both $\Val{\sigma}$ and $\Res{\sigma}$ are.
\end{remark}

We call $(\RV_K,\oplus,\cdot,0,1,\Rv{\sigma})$ the \emph{leading term (difference) structure} of $(K,\val,\sigma)$.

\runin{Languages and theories} We now introduce the language that we will use to prove relative quantifier elimination.
\begin{definition}\label{def:three-sorts-with-lambda}
    We let $\Lvfe$ be the three-sorted language whose sorts are given as follows:
    \begin{enumerate}
        \item $\mainsort$ is the main sort, with language $\LK = \{+, \cdot, -, 0, 1, \sigma\}$,
        \item $\RVsort$ has language $\L_{\RVsort} \coloneqq \{\oplus, \cdot, 0, 1, \Rv{\sigma}\}$,
        \item $\valuesort$ has language $\Loagsigma \coloneqq \{+,\leq,0,\infty,\Val{\sigma}\}$,
    \end{enumerate}
    with functions given by $\rv \colon \mainsort \to \RVsort$ and $\val_{\rv} \colon \RVsort \to \valuesort$. We let $\Lsorts$ be the reduct of $\Lvfe$ to the sorts $\RVsort$ and $\valuesort$.
\end{definition}
\begin{remark}\label{interpret-value}
    Note that one can define $\val$ in $\Lvfe$ as $\val = \val_{\rv} \circ \rv$. Moreover, in any valued difference field $(K,\val,\sigma)$, $(\Gamma_K,+,0,\leq,\infty,\Val{\sigma})$ and $(\residue{K},+,\cdot,-,0,1,\Res{\sigma})$ are interpretable in $(\RV_K,\oplus,\cdot,0,1,\Rv{\sigma})$.
\end{remark}
\begin{definition}\label{lksep}
    We let $\LKsep$ be the expansion of $\LK$ where we adjoin, for every $n \geq 1$, $i \in \{1, \dots n\}$, an $(n+1)$-ary function $\lambda_n^i$. We let $\Lsep$ be the expansion of $\Lvfe$ where we give $\mainsort$ the language $\LKsep$.
\end{definition}
\begin{definition}\label{lambda-sigma}
    We let $\Tsep$ be the $\LKsep$-theory of difference fields extending $\primary$ where we interpret $\lambda_n^i$ as follows. Let $x_1, \dots x_n, y \in K$, and assume that $x_1, \dots x_n$ are $\sigma(K)$-linearly independent, with $y \in \operatorname{span}_{\sigma(K)}(x_1, \dots x_n)$. Then, $\lambda_n^1(\overline{x},y) \dots \lambda_n^n(\overline{x},y)$ are the unique elements of $K$ such that 
    \[
        y = \sum_{i=1}^n \sigma(\lambda_n^i(x_1, \dots x_n, y)) x_i.
    \]
    Otherwise, set $\lambda_n^1(x_1, \dots x_n, y) = \cdots = \lambda_n^n(x_1, \dots x_n, y) = 0$.
\end{definition}
\begin{remark}\label{language-separability}
    If $(L,\sigma) \vDash \Tsep$ and $(K,\sigma) \subseteq (L,\sigma)$ is a difference subfield, then in $\LKsep$ we have that ${(K,\sigma) \leq (L,\sigma)}$ if and only if the extension is $\sigma$-separable. Indeed, $(K,\sigma)$ is closed under the $\lambda$-functions precisely if and only if $K$ is linearly disjoint from $\sigma(L)$ over $\sigma(K)$.
\end{remark}
\begin{definition}
    We let $\Tvdf$ be the $\Lsep$-theory whose models $\model{K} = \langle K,\RV_K, \valuegroup{K} \rangle$ are non-inversive valued difference fields with leading terms difference structure $\RV_K$, value difference group $\valuegroup{K}$, $\Rv{\sigma}$, and further $(K,\sigma)$ is a model of $\Tsep$. We let $\Tvdfzero$ extend $\Tvdf$ by further requiring that models have residue characteristic zero.
\end{definition}
\begin{definition}
    We let $\WTH$ be the $\Lsep$-theory extending $\Tvdf$ whose models $(K,\val,\sigma)$ are non-inversive weakly $\sigma$-henselian and such that $\sigma(K)$ is relatively algebraically closed in $K$ (i.e., they are also models of $\primary$). We let $\WTHzero$ extend $\WTH$ by further requiring that models have residue characteristic zero.
\end{definition}

\begin{remark}
    Note that an inversive model of $\WTHzero$ is precisely an equicharacteristic zero strongly $\sigma$-henselian valued difference field, as considered by \cite{onay2013quantifier}. We also point out that we don't need to impose, when defining $\WTHzero$, that $\Res{\sigma}$ and $\Val{\sigma}$ are surjective, as this follows from weak $\sigma$-henselianity by \cref{hens-implies-inv}.
\end{remark}

\par We now prove a somewhat surprising lemma about the behaviour of $\lambda$-functions on subrings of our models; this mimics the analogous phenomenon in separably closed valued fields.

\begin{lemma}\label{lambda-subfield}
    Let $\model{K} \vDash \WTH$ and $\model{A} \leq \model{K}$. Then, there is a unique $\Lsep$-structure on $\operatorname{Frac}(A)$ extending the one on $A$. In other words, given $\model{A} \leq \model{K}$, we may always assume that $A$ is a field.
\end{lemma}
\begin{proof}
    As $\rv$ is multiplicative, it extends uniquely to $\operatorname{Frac}(A)$. The $\lambda$-functions also extend uniquely to $\operatorname{Frac}(A)$: indeed, if $A \subseteq \core{K}$, then $\operatorname{Frac}(A) \subseteq \core{K}$, and the $\lambda$-functions are trivial. Otherwise, assume that $A \not\subseteq \core{K}$: we show that $A = \operatorname{Frac}(A)$. Indeed, if $A \not\subseteq \core{K}$ then by induction we may assume that there is ${a \in A \setminus \sigma(K)}$, and so we can compute $\lambda_1^1(\sigma(a),1) = \frac{1}{a} \in A$. 
    For any $b \in \sigma(K) \cap A$, we have $ab \notin \sigma(K)$, and thus by our previous considerations $\frac{1}{ab} \in A$, so $\frac{1}{b} \in A$.
\end{proof}

\runin{Building models of $\WTHzero$} We exhibit a recipe for building non-inversive models with prescribed residue difference field and value difference group. The standard Hahn constructions don't work, for reasons clarified in \cref{hahndoesntwork}.

\begin{definition}
    Let $(K,\sigma)$ be a difference field. We say that it is \emph{weakly linearly difference closed} if for every $b, a_0, \cdots a_n \in K$ with $a_0 \neq 0$, there is $x \in K$ with
    \[
        b + a_0 x + a_1 \sigma(x) + \cdots + a_n \sigma^n(x) = 0.
    \]
\end{definition}

\begin{remark}
    If $(K,\sigma)$ is weakly linearly difference closed, then $\inv{K}$ is linearly difference closed.
\end{remark}

\par Let $(k,\Res{\sigma})$ be a weakly linearly difference closed, non-inversive difference field of characteristic zero, and $(\Gamma,\Val{\sigma})$ a non-inversive ordered difference group.

Consider the generalized power series field $K \coloneqq k(\!(\Gamma)\!)$ with the $t$-adic valuation $\val$ and the lift of $\Res{\sigma}$ and $\Val{\sigma}$ given by
    \[
        \sigma\left(\sum_{\gamma \in \Gamma} a_\gamma t^\gamma\right) = \sum_{\gamma \in \Gamma} \Res{\sigma}(a_\gamma) t^{\Val{\sigma}(\gamma)}.
    \]
    Let $L \coloneqq \inv{k}(\!(\inv{\Gamma})\!)$. There is a natural embedding $K \subseteq L$ (but $L$ is not the inversive closure of $K$), and $L$ satisfies $\primary$ since it is inversive. Then,

\begin{proposition}\label{models}
    The relative $\sigma$-separably $\sigma$-algebraic closure $\widetilde{\FE{K}}$ of $\FE{K}$ in $L$ is a model of $\WTHzero$. It has residue difference field $(\inv{k},\Res{\sigma})$ and value difference group $(\inv{\Gamma},\Val{\sigma})$.
\end{proposition}
\begin{proof}
    Since $k$ is weakly linearly difference closed, $\inv{k}$ is linearly difference closed, and so (by \cref{maximality-hens}) $L$ is strongly $\sigma$-henselian. Then, by construction, $\widetilde{\FE{K}}$ is weakly $\sigma$-henselian. In particular, it has inversive residue difference field and inversive value difference group by \cref{hens-implies-inv}, thus they are equal to $\inv{k}$ and $\inv{\Gamma}$. Since it is a $\sigma$-separable extension of $k(\!(\Gamma)\!)$, which is not inversive, it is still not inversive.
\end{proof}

\begin{remark}\label{hahndoesntwork}
    As $(k(\!(\Gamma)\!),\val)$, as defined above, is maximal as a valued field (i.e. it has no proper immediate extension), then whenever $\sigma \in \operatorname{End}(k(\!(\Gamma)\!),\val)$, we have that $\sigma$ is surjective if and only if both $\Res{\sigma}$ and $\Val{\sigma}$ are. Indeed, if both $\Res{\sigma}$ and $\Val{\sigma}$ are surjective, $\sigma(k(\!(\Gamma)\!)) \subseteq k(\!(\Gamma)\!)$ is an immediate extension, and thus it must be trivial. In particular, this means that there is no non-inversive model of $\WTH$ of the form $(k(\!(\Gamma)\!),\val,\sigma)$.
\end{remark}

\runin{Embeddings in models} We explain how $\sigma$-separably $\sigma$-algebraically maximal $\sigma$-separably $\sigma$-algebraic immediate extensions can be embedded in saturated models.

\begin{fact}[{\cite[Corollary 5.10]{onay2013quantifier}}]
    \label{lem:embedding-inv}
    Let $(K,\val,\sigma)$ be an inversive valued difference field with linearly difference closed residue difference field. Let $(K',\val,\sigma)$ be a $\vert \valuegroup{K}\vert^+$-saturated strongly $\sigma$-henselian extension of $K$. Then, any $\sigma$-algebraically maximal immediate $\sigma$-algebraic extension of $K$ embeds into $K'$ over $K$.
\end{fact}
\begin{lemma}
    \label{embed-in-saturated}
    Suppose $(K,\val,\sigma)$ is a model of $\primary$ and $(\residue{K},\Res{\sigma})$ is linearly difference closed. 
    
    Let $(K,\val,\sigma) \subseteq (K^*,\val,\sigma)$ be a $\sigma$-separable extension, where $(K^*,\val,\sigma)$ is a model of $\primary$, is $\vert \valuegroup{K} \vert^+$-saturated and weakly $\sigma$-henselian. Then there is a $K$-embedding $f\colon \smax{K} \to K^*$ such that $K^*$ is $\sigma$-separable over $f(\smax{K})$.
\end{lemma}

\begin{proof}
    Consider a $\vert K \vert^+$-saturated (and thus in particular $\vert \valuegroup{K} \vert^+$-saturated) elementary extension $L$ of $\inv{(K^*)}$. Then, by \cref{inv-strongly-hens} and \cref{lem:embedding-inv}, $\inv{(\smax{K})}$ embeds into $L$ over $\inv{K}$, in particular over $K$, along a map $\theta$. Now, by saturation, it is enough\footnote{This clever trick is taken from \cite[Proposition 5.8]{dor2023contracting}.} to exhibit a witness in $K^*$ for every quantifier-free $\Lvfe$-formula $\varphi$ over $K$ such that, for some finite tuple $\alpha = (a_1, \dots a_r)$ in $\smax{K}$, $\smax{K} \vDash \varphi(\alpha)$. Upon strenghtening $\varphi$, we may assume all of its solutions are simple roots of difference polynomials over $K$.
    \par The tuple $(\theta(a_1), \dots \theta(a_r)) \in L^r$ satisfies $\varphi$. As we have taken $\inv{(K^*)} \preceq L$, then there exist elements ${b_1, \dots b_r \in \inv{(K^*)}}$ such that $\inv{(K^*)} \vDash \varphi(b_1, \dots b_r)$. But $b_1, \dots b_r$ are simple roots of difference polynomials over $K$, in particular over $K^*$, and thus they are in $K^*$ already, as $K^*$ is relatively $\sigma$-separably $\sigma$-algebraically closed in its inversive hull.
\end{proof}

\runin{Adding residues}
We establish the two ways of increasing the residue difference field in a back-and-forth situation. Note that the proofs of the following two lemmas from \cite{azgin2011elementary} only rely on $\Res{\sigma}$ and $\Val{\sigma}$ being surjective, and thus apply also in our context. We let $(K,\val,\sigma) \subseteq (L,\val,\sigma)$ be a valued difference field extension.
\begin{fact}[{\cite[Lemma 2.5]{azgin2011elementary}}]\label{adjoin-transc}
    Let $a \in \O_{\core{L}}$ and assume that $\alpha = \res(a)$ is $\sigma$-transcendental over $\residue{K}$. Then,
    \begin{enumerate}
        \item $\val(p(a)) = \min_{I}\val(b_{I})$, for each $p(X) = \sum_{I} b_{I} X^{I}$ over $\O_K$,
        \item $\ext{K}{a}$ has residue difference field $\ext{\residue{K}}{\alpha}$ and value difference group $\valuegroup{K}$,
        \item if $b$ is in $\O_{\core{L'}}$, for some extension $L'$ of $K$, such that $\res(b)$ is $\sigma$-transcendental over $\residue{K}$, then there is a valued difference field isomorphism $\ext{K}{a} \to \ext{K}{b}$ over $K$, sending $a$ to $b$.
    \end{enumerate}
\end{fact}
\begin{fact}[{\cite[Lemma 2.6]{azgin2011elementary}}]\label{adjoin-alg}
    Let $a \in \O_{L}$ and assume that $\alpha = \res(a)$ is $\sigma$-algebraic over $\residue{K}$. Let $\res{p}(X) \in \polring{\residue{K}}$ be a difference polynomial of minimal complexity such that $\res(p)(\alpha) = 0$, and let $p(X)$ have the same complexity and be such that $p(a) = 0$. Then,
    \begin{enumerate}
        \item $\extnot{K}{a}$ has residue difference field $\extnot{\residue{K}}{\alpha}$ and value difference group $\valuegroup{K}$,
        \item if $b$ is in $\O_{L'}$, for some extension $L'$ of $K$, such that $\res{p}$ is a difference polynomial of minimal complexity such that $\res(p)(\res(b)) = 0$, with $p(b) = 0$, then there is a valued difference field isomorphism $\extnot{K}{a} \to \extnot{K}{b}$ over $K$, sending $a$ to $b$.
    \end{enumerate}
\end{fact}
\begin{remark}
    Note that, since $(\residue{K},\Res{\sigma})$ is inversive, by exchanging $p(X)$ with the difference polynomial $\shifted{p}^{\sigma^{-m}}(X)$ defined in \cref{shift-reference} we can always choose $p(X)$ such that $p' \neq 0$. Thus, if $(K,\val,\sigma)$ is a model of $\primary$, the resulting extension is $\sigma$-separably $\sigma$-algebraic over $K$. Upon passing to the $\primary$-closure (thanks to \cref{primary-closure}), we may even assume that the extension is a model of $\primary$, but it might not be generated by one element anymore.
\end{remark}

\runin{Regular elements} There are several ways to increase the value difference group; here we take the same route as \cite{onay2013quantifier}, which goes via adjoining a very generic element (as opposed to, for example, the route taken by \cite{rideau2017some}).

\begin{definition}
    Let $a \in K$ and $p(X) = \sum_{I} a_{I} X^{I} \in \polring{K}$. We say that \emph{$a$ is regular for $p$} if $\val(p(a)) = \min_{I} \{\val(a_{I})+\val(a^{I})\}$.
    Let $(K,\val,\sigma) \subseteq (L,\val,\sigma)$ be an extension of valued difference fields. We say that $a \in L$ is \emph{generic over $K$} if it is regular for every $p \in \polring{K}$.
\end{definition}

\begin{proposition}[\nonsurj]\label{generics-exist}
    Let $\model{K} \vDash \WTHzero$ and let $(E,\val,\sigma) \subseteq (K,\val,\sigma)$ be a valued difference subfield, possibly with non-inversive residue difference field and value difference group. Assume that $\model{K}$ is $\max(\vert \residue{E}\vert^+,\aleph_0)$-saturated, and let $\gamma \in \valuegroup{K}$. Then, there is $a \in \core{K}$, generic over $E$, such that:
    \begin{enumerate}
        \item $\val(a) = \gamma$,
        \item $\ext{E}{a}$ has value group $\ext{\valuegroup{E}}{\gamma}$,
        \item if $b$ is another generic over $E$ in some extension $(L,\val,\sigma)$, with $\val(b) = \gamma$ and $b \in \core{L}$, then the map $a \mapsto b$ gives rise to a valued difference field isomorphism $\ext{E}{a} \to \ext{E}{b}$ over $E$.
    \end{enumerate}
    If, moreover, $(E,\sigma) \subseteq (K,\sigma)$ is $\sigma$-separable, then $\ext{E}{a} \subseteq K$ is still $\sigma$-separable. Analogously, if $(E,\sigma) \subseteq (L,\sigma)$ is $\sigma$-separable, then $\ext{E}{b} \subseteq L$ is still $\sigma$-separable.
\end{proposition}

\begin{remark}
    Note that this is essentially a twisted version of \cref{adjoin-transc}, although it is from a different point of view; \cref{adjoin-transc} shows how the isomorphism type is uniquely determined for a generic with $\gamma = 0$.
\end{remark}

\begin{proof}[Proof of \cref{generics-exist}]
    Under these assumptions, $\core{K} \subseteq K$ is dense (\cref{core-dense}), and thus immediate.
    By $\vert \residue{E} \vert^+$-saturation, there is $\beta \in \residue{K}$ which is $\sigma$-transcendental over $\residue{E}$. We let $b \in \O_{\core{K}}$ be such that $\res(b) = \beta$. Then, by construction, $b$ is a generic of value zero. If $c \in \core{K}$ is such that $\val(c) = \gamma$, then, we claim that $a \coloneqq cb$ is still a generic. Indeed, if $p(X) = \sum_{I} a_{I} X^{I} \in \polring{E}$, we find $d \in \O_K$ such that $p(cX) = dq(X)$, with $\val(d) = \min_{I}(\val(a_{I}) + \val(c^{I}))$ and $q(X) \in \polring{\O_{K}} \setminus \polring{\m_K}$. By construction, $\val(q(b)) = 0$, and thus
    \[
        \val(p(a)) = \val(p(cb)) = \val(d) + \val(q(b)) = \val(d),
    \]
    as required.
    The rest of the properties then follow (arguing as in the proof of \cref{adjoin-transc}) from the fact that if $a$ and $b$ are generics over $E$, then for every $p(X) = \sum_I a_I X^I \in \polring{E}$ we can compute 
    \[
        \val(p(a)) = \min_I \{\val(a_I) + \val(a^I)\} = \min_I \{\val(a_I) + I(\gamma)\} = \val(p(b)),
    \]
    and thus the difference field isomorphism $\ext{E}{a} \to \ext{E}{b}$ can be upgraded to a valued field isomorphism.
\end{proof}
\begin{remark}
    \cref{generics-exist} is proved essentially in the same way as \cite[Lemma 6.1]{onay2013quantifier}. Note, however, that the statement in that paper is incorrect: namely, it states that $\ext{E}{a}$ has the same residue difference field as $E$. This is false: if, for example, $\gamma = 0$, then we would be extending the residue difference field with a $\sigma$-transcendental element (the residue of the generic). This does not create issues in the surjective case, but in our setting it is precisely the reason why we need \cref{inversive-lift}.
\end{remark}
\par We next prove that genericity of an element $b$ is really a property of $\rv(b)$.

\begin{lemma}\label{generics-rv}
    Let $\model{K} \vDash \WTHzero$ and let $\model{A} \leq_{\Lsep} \model{K}$ be a valued difference subfield. Let $b \in K$ be generic over $A$, and let $b' \in K$ be such that $\rv(b') = \rv(b)$. Then $b'$ is generic over $A$.
\end{lemma}
\begin{proof}
    Let $p(X) = \sum_{I} a_{I} X^{I} \in \polring{A}$. Since $b$ is generic over $A$,
    \[
        \val(p(b)) = \min_{I} (\val(a_{I})+\val(b^{I})),
    \]
    and thus 
    \[
        \rv(p(b)) = \bigoplus_{I} \rv(a_{I}b^{I}) = \bigoplus_{I} \rv(a_{I}) \rv(b)^{I},
    \]
    so 
    \[
        \rv(p(b')) = \bigoplus_{I} \rv(a_{I})\rv(b')^{I} = \bigoplus_{I} \rv(a_{I}(b')^{I})
    \]
    is well-defined, i.e. $\val(p(b')) = \min_{I}(\val(a_{I})+\val((b')^{I}))$.
\end{proof}

\runin{Auxiliary surjectivity} In general, it is not clear if one can ensure that if $E$ has inversive residue field, then $\ext{E}{a}$ also does. In particular, the residue field might turn non-inversive after adding a generic to $E$ via \cref{generics-exist}.

\begin{example}
    Let $(E,v,\sigma) \subseteq (K,v,\sigma)$ be valued difference fields, with $\Val{\sigma}$ $\omega$-increasing, i.e. for all $\gamma > 0$ and $n > 0$, $\Val{\sigma}(\gamma) > n\gamma$. Assume further that $E$ is not inversive, but $\residue{E}$ is. Let $b \in K$ be generic over $E$, and assume that there is $c \in E \setminus \sigma(E)$ such that, for some $\ell \geq 0$, $\val(c) = \Val{\sigma}^\ell(\val(b))$. Without loss of generality, take $\ell = 0$, i.e. $\val(c) = \val(b)$. Then, $\frac{b}{c}$ is still a generic, thus $\extnot{E}{b} = \extnot{E}{\frac{b}{c}}$ has residue difference field $\extnot{\residue{E}}{\res\left(\frac{b}{c}\right)}$, generated by an element $\sigma$-transcendental over $\residue{E}$; in particular, it is non-inversive. We now argue that $\ext{E}{b}$ also has non-inversive residue difference field. Indeed, consider $\extnot{E}{b}(\sigma^{-1}(b))$: since $\Val{\sigma}$ is $\omega$-increasing, this is an extension by a transcendental element with value outside of the divisible hull of $\valuegroup{\extnot{E}{b}}$. In particular, the residue field does not change. We can see, by induction, that $\ext{E}{b}$ has residue field given by $\extnot{\residue{E}}{\res\left(\frac{b}{c}\right)}$, which is non-inversive. 
\end{example}

This issue is fixed by moving to a residually inversive hull, as we explain now.

\begin{lemma}[\nonsurj]\label{inversive-lift}
    Let $(K,\val,\sigma)$ be a weakly $\sigma$-henselian valued difference field that is also a model of $\primary$. Let $(F,\val,\sigma)$ be a valued difference subfield such that $(F,\sigma) \subseteq (K,\sigma)$ is $\sigma$-separable, and $(\residue{F},\Res{\sigma})$ is not necessarily inversive. Then, there is an extension $F \subseteq F' \subseteq K$ with the following properties:
    \begin{enumerate}
        \item $F \subseteq F'$ is $\sigma$-separably $\sigma$-algebraic,
        \item $F'$ is a model of $\primary$, and thus $F' \subseteq K$ is $\sigma$-separable,
        \item the residue difference field of $F'$ is $\inv{\residue{F}}$,
        \item if $(L,\val,\sigma)$ is another $\sigma$-separable extension of $K$ which is a model of $\primary$ and is weakly $\sigma$-henselian, then $F'$ embeds into $L$ over $K$ so that $L$ is $\sigma$-separable over the image\footnote{Note that the embedding in question is not necessarily unique.}.
    \end{enumerate}
    If, further, $(\valuegroup{F},\Val{\sigma})$ is inversive, then $\valuegroup{F} = \valuegroup{F'}$.
\end{lemma}
\begin{proof}
    We build $F'$ recursively, alongside its $F$-embedding into $L$. Both $(K,\val,\sigma)$ and $(L,\val,\sigma)$ are in particular henselian, so we might replace $F$ with $\FE{F^h}$ and assume that $F$ is henselian as well. Note that $F \subseteq F^h$ is an algebraic extension, in particular because $F$ is a model of $\primary$, it is ($\sigma$-)separable, thus the same holds for $F \subseteq \FE{F^h}$ by \cref{primary-equivalence}. Then, $\FE{F^h} \subseteq K$ is $\sigma$-separable. If $(\valuegroup{F},\Val{\sigma})$ is inversive, then $F \subseteq \FE{F^h}$ is a purely residual extension, since $\FE{F^h} \subseteq \inv{(F^h)}$.
    \par We may thus assume that $(F,\val)$ is henselian and is a model of $\primary$. Then, by \cite[Lemma 5.9]{dor2023contracting}, $(\residue{F},\Res{\sigma})$ is a model of $\primary$ as well. In particular, every ${\alpha \in \residue{F} \setminus \Res{\sigma}(\residue{F})}$ is transcendental over $\Res{\sigma}(\residue{F})$.
    \par Let $\alpha \in \residue{F}$ be such that $\alpha \notin \Res{\sigma}(\residue{F})$. Given any $\epsilon \in \m_F$ and any $a \in \O_F$ such that $\res(a) = \alpha$, consider $f(X) = \sigma(X) - \epsilon X - a$. By \cref{as-deep}, there is $b \in \O_K$ such that $\sigma(b)-\epsilon b - a = 0$ and so $\Res{\sigma}(\res(b)) = \alpha$. Note that $\sigma(X)-\alpha$ is of minimal complexity for $\res(b)$ over $\residue{F}$, thus by repeating the same argument symmetrically in $L$ we obtain, by \cref{adjoin-alg}, an embedding $\ext{F}{b} \to L$ such that $L$ is $\sigma$-separable over its image. Note that $b$ is transcendental over $F$, and moreover $\ext{F}{b} = F(b)$, thus $\ext{F}{b}$ is a model of $\primary$ (as $\sigma(\ext{F}{b}) = \sigma(F)(\epsilon b-a) \subseteq F(\epsilon b-a) = F(b)$ is a regular extension). As $F \subseteq \ext{F}{b}$ is $\sigma$-separably $\sigma$-algebraic, it follows that $\ext{F}{b} \subseteq K$ is still $\sigma$-separable. Moreover, as $\residue{F} \subseteq \residue{F}(\res(b))$ and $F \subseteq F(b)$ are both purely transcendental extensions, the valuation is uniquely determined to be the Gauss valuation, and hence $\valuegroup{F(b)} = \valuegroup{F}$.
    \par Upon replacing $\ext{F}{b}$ with $\FE{\ext{F}{b}^h}$ again, we may restart the process, whose limit is the required $F'$ together with an $F$-embedding into $L$.
\end{proof}

\begin{definition}
    We call $F'$ as in \cref{inversive-lift} a \emph{residually inversive hull} of $F$.
\end{definition}

%% file: embedding-and-consequences.tex
We now have all the tools necessary to prove the embedding lemma, and then deduce relative quantifier elimination and its consequences.

\runin{The embedding lemma} Unless otherwise stated, when we write $\model{A} \leq \model{K}$ we mean that $\model{A}$ is an $\Lsep$-substructure of $\model{K}$ (i.e., if $A$ is a valued difference subfield, that the extension $(A,\sigma) \subseteq (K,\sigma)$ is $\sigma$-separable). When we write an embedding $f \colon \model{A} \to \model{L}$, we mean that it is an $\Lsep$-embedding (i.e., if $A$ is a valued difference subfield, that $L$ is $\sigma$-separable over $f(A)$), and we denote by $f$ the embedding on $\mainsort(\model{A})$, by $f_{\RVsort}$ the induced embedding on $\RVsort(\model{A})$, and by $f_{\valuesort}$ the induced embedding on $\valuesort(\model{A})$. If $\theta_\RV \colon \RV_A \to \RV_L$ and $\theta_\Gamma\colon \valuegroup{A} \to \valuegroup{L}$ are embeddings, we call them \emph{compatible} if the diagram
\[
    \begin{tikzcd}
        {\RVsort(A)} && {\valuesort(A)} \\
        \\
        {\RVsort(L)} && {\valuesort(L)}
        \arrow["{\val_{\rv}}", from=1-1, to=1-3]
        \arrow["{\theta_\RV}"', from=1-1, to=3-1]
        \arrow["{\theta_\Gamma}", from=1-3, to=3-3]
        \arrow["{\val_{\rv}}"', from=3-1, to=3-3]
    \end{tikzcd}
\]
commutes.

\begin{lemma}[Embedding lemma]\label{embedding-theorem}
    Let $\model{K}$ and $\model{L}$ be $\aleph_1$-saturated models of $\WTHzero$.
    
    Let:
    \begin{enumerate}
        \item $\model{A} \leq \model{K}$ be a countable substructure,
        \item $f \colon \model{A} \hookrightarrow \model{L}$ be an embedding,
        \item $\theta_\RV \colon \RV_K \hookrightarrow \RV_L$ and $\theta_\Gamma \colon \valuegroup{K} \hookrightarrow \valuegroup{L}$ be compatible embeddings over $\RVsort(\model{A})$ and $\valuesort(\model{A})$, extending $f_{\RVsort}$ and $f_{\valuesort}$.
    \end{enumerate}
    
    Then, for every $a \in K$, there is a countable substructure $\model{A} \leq \model{A'} \leq \model{K}$ together with an embedding $g \colon \model{A'} \hookrightarrow \model{L}$ extending $f$ such that $a \in A'$, $g_{\valuesort} = \theta_\Gamma\vert_{\valuesort(\model{A}')}$, and $g_{\RVsort} = \theta_{\RV}\vert_{\RVsort(\model{A}')}$.
\end{lemma}

\begin{figure*}
\begin{tikzcd}
	{(\RV_K,\valuegroup{K})} && {(\RV_L,\valuegroup{L})} \\
	{{\color{red}(\RVsort(A'),\valuesort(A'))}} \\
	& {(\RV_A,\valuegroup{A})} \\
	K && L \\
	{\color{red} A'} \\
	& {\extnot{A}{a}} \\
	& A
	\arrow["{(\theta_\RV,\theta_\Gamma)}"{description}, from=1-1, to=1-3]
	\arrow[color={rgb,255:red,214;green,92;blue,92}, no head, from=2-1, to=1-1]
	\arrow["{(g_\RVsort\vert_{\RVsort(\model{A'})},g_\valuesort\vert_{\valuesort(\model{A'})})}"{description}, color={rgb,255:red,214;green,92;blue,92}, hook, from=2-1, to=1-3]
	\arrow[no head, from=3-2, to=1-1]
	\arrow["{(f_\RVsort, f_\valuesort)}"{description}, hook, from=3-2, to=1-3]
    \arrow[color={rgb,255:red,214;green,92;blue,92}, no head, from=3-2, to=2-1]
	\arrow[color={rgb,255:red,214;green,92;blue,92}, no head, from=5-1, to=4-1]
	\arrow["g"{description}, color={rgb,255:red,214;green,92;blue,92}, hook, from=5-1, to=4-3]
	\arrow[color={rgb,255:red,214;green,92;blue,92}, no head, from=6-2, to=5-1]
	\arrow["f"{description}, hook, from=7-2, to=4-3]
	\arrow[no head, from=7-2, to=6-2]
\end{tikzcd}
\caption{The red elements of the diagram are the ones produced by \cref{embedding-theorem}.}
\end{figure*}
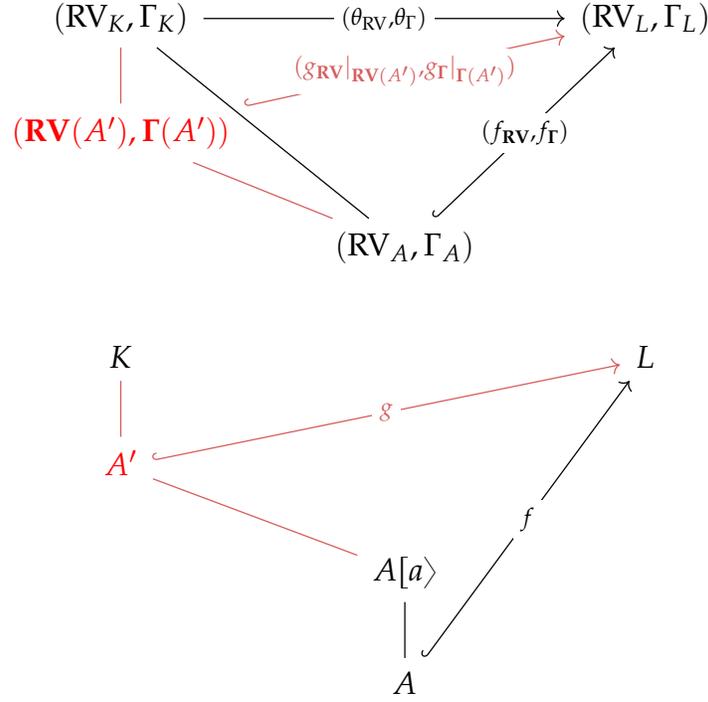

\begin{remark}
    Note that this does not immediately yield that we can lift $\theta_\RV$ and $\theta_\Gamma$ to an embedding of $\model{K}$ itself in $\model{L}$; for this, $\model{K}$ would need to be of size $\aleph_1$, i.e. saturated in its own cardinality, which could be achieved for example by assuming the Continuum Hypothesis (this would anyway be a safe assumption, see \cite{halevi2023saturated}).
\end{remark}

\begin{remark}
    The embedding lemma is phrased in a bit of an odd way, to allow space for both relative quantifier elimination and \cref{existential-ake}. As explained in the previous remark, this falls short of a purely algebraic embedding lemma due to saturation issues, essentially injected into the statement by the use of generics. One should think of applying this to an asymmetric back-and-forth scenario, where $\model{L}$ is saturated in the cardinality of $\model{K}$, $f_{\RVsort}$ and $f_{\valuesort}$ are elementary, and thus they then be extended by saturation and elementarity to $\theta_\RV$ and $\theta_\Gamma$ in a compatible way. The embedding lemma then gives a recipe to extend $f$ to any element of $K$ which is not in $A$ already, possibly extending the residue field along the way. In fact, this remark is a spoiler: this is precisely how the theorem will be used in \cref{relative-qe}.
\end{remark}

\begin{proof}[Proof of \cref{embedding-theorem}]
    We proceed in steps, extending $f$ recursively and checking at each stage that we still obtain an embedding as requested.

    \par \textbf{Warning.} At each stage, we rename the newly obtained intermediate substructure as $\model{A}$ again, and the extended embedding as $f$ again, to save on notation. From the moment we can assume that we are working with subfields onwards, we check that $\model{A} \leq \model{K}$ by checking that $K$ is $\sigma$-separable over $A$; similarly, we check that a valued difference field embedding $f \colon A \to L$ is an embedding in $\Lsep$ by checking that $L$ is $\sigma$-separable over $f(A)$. We also point out that, in this setting, the towers of extension $\core{K} \subseteq K \subseteq \inv{K}$ and $\core{L} \subseteq L \subseteq \inv{L}$ are all dense, and in particular immediate. We make use of this fact several times without further mention in the proof.

    \begin{itemize}

        \item[\text{\textsc{step 0.}}] \textit{We may assume that $A$ is a field.} \\
        By \cref{lambda-subfield}, $f$ extends uniquely to $\Frac(A)$. \hfill $\blacksquare^0$

        \item[\text{\textsc{step 1.}}] \textit{We may assume that $A$ satisfies $\primary$.} \\
        By \cref{primary-closure}, since $K$ is $\sigma$-separable over $A$ and satisfies $\primary$, $A$ must be a model of $\primary$ as well. \hfill $\blacksquare^1$

    \end{itemize}

    \par We are now working with a valued difference subfield $(A,\val,\sigma)$. Let $A'$ be the smallest $\Lvfe$-substructure of $\model{A}$ containing $A$ and $a$. Note that, a priori, $A'$ will not be equal to $\extnot{A}{a}$, unless the extension $\extnot{A}{a} \subseteq K$ happens to $\sigma$-separable.

    \begin{itemize}
        \item[\text{\textsc{step 2.}}] \textit{We may assume that for every $z \in A'$, $\val(z) \in \valuegroup{A}$.}\\
        Let $\gamma \in \valuegroup{A'} \setminus \valuegroup{A}$, and let $b \in \core{K}$ be generic over $A$ with $\val(b) = \gamma$ (such $b$ exists by \cref{generics-exist}). Let $\alpha = \rv(b)$ and $\alpha' = \theta_\RV(\alpha) \in \RV_L$. We now let $b' \in \core{L}$ be such that $\rv(b') = \alpha'$ (in particular, $\val(b') = \theta_\Gamma(\val(b))$). Then, by \cref{generics-rv}, $b'$ is still generic over $f(A)$, thus \cref{generics-exist} ensures an isomorphism $\ext{A}{b} \to \ext{f(A)}{b'}$, extending $f$ and $b \mapsto b'$, where $\ext{f(A)}{b'} \subseteq L$ is still $\sigma$-separable, and thus this gives rise to an $\Lsep$-embedding of $\ext{A}{b}$ into $L$ over $A$. Moreover, since $\ext{A}{b} \subseteq K$ and $\ext{A}{b'} \subseteq L$ are $\sigma$-separable (as $b \in \core{K}$ and $b' \in\core{L}$), using \cref{primary-equivalence} and since both $(K,\sigma)$ and $(L,\sigma)$ are models of $\primary$, one gets automatically that $\ext{A}{b}$ and $\ext{A}{b'}$ also do. \\
         We can then repeat this procedure $\omega$-many times to obtain a new countable substructure $A_1$ such that whenever $z \in A'$, we have that $\val(z) \in \valuegroup{A_1}$. Note that, a priori, we do not have that whenever $z \in \extnot{A_1}{a}$, $\val(z) \in \valuegroup{A_1}$. Thus, we repeat this procedure countably many times to produce a chain of extensions $A = A_0 \leq A_1 \leq A_2 \leq \cdots$ such that their union, which is again countable and we call $A_\infty$, now satisfies that whenever $z \in \extnot{A_\infty}{a}$, $\val(z) \in \valuegroup{A_\infty}$. The procedure might have made $\residue{A_\infty}$ non-inversive, so we apply \cref{inversive-lift} to replace $A_\infty$ with a residually inversive hull. We rename $A_\infty$ as $A$. \hfill $\blacksquare^2$
         
         \item[\text{\textsc{step 3.}}] \textit{We may assume that for every $z \in A'$, $\res(z) \in \residue{A}$, and that $(\residue{A},\Res{\sigma})$ is linearly difference closed.}\\ 
         Identify $\residue{A} \subseteq \RV_A$, $\residue{L} \subseteq \RV_L$, and $\residue{K} \subseteq \RV_K$. Let $\alpha \in \residue{A'} \setminus \residue{A}$, and let $\alpha' = \theta_\RV(\alpha) \in \RV_L$. We distinguish two sub-cases.\\
         \textit{Subcase 3.a:} $\alpha$ is $\sigma$-transcendental over $\residue{A}$.\\
         Then, the same is true for $\alpha'$, so if we choose any $b \in \core{K}$ with $\rv(b) = \res(b) = \alpha$ and any $b' \in \core{L}$ with $\rv(b') = \res(b') = \alpha'$, by \cref{adjoin-transc} the map $b \mapsto b'$ extends $f$ to an $\Lsep$-embedding (since both $\ext{A}{b} \subseteq K$ and $\ext{A}{b'} \subseteq L$ are $\sigma$-separable, by \cref{inversive-relative-transc}) $f\colon \ext{A}{b} \to \model{L}$ over $A$.\\
         \textit{Subcase 3.b:} $\alpha$ is $\sigma$-algebraic over $\residue{A}$.\\
         Then, we let $h(X) \in \polring{\residue{A}}$ be of minimal complexity such that $h(\alpha) = 0$. As $(\residue{A},\Res{\sigma})$ is inversive, we have that $h' \neq 0$ and, moreover, that whenever $h_{J} \neq 0$, $h_{J}(\alpha) \neq 0$. We now let $c \in \O_{K}$ be such that $\res(c) = \alpha$, and let $g(X) \in \polring{\O_A}$ be an exact lift of $h$. Then, $g' \neq 0$, $\val(g_{J}(c)) = 0$ for all $J$ such that $g_{J} \neq 0$, and $\val(g(c)) > 0$. As $K$ is weakly $\sigma$-henselian, we find $b \in K$ with $\res(b) = \alpha$ and $g(b) = 0$. Analogously, on the other side, we find $b' \in L$ with $\res(b') = \alpha'$ and $g(b') = 0$.\\
         By \cref{adjoin-alg}, this gives rise to an isomorphism $\extnot{A}{b} \to \extnot{f(A)}{b'}$, given by $b \mapsto b'$, which extends to an isomorphism $\FE{\extnot{A}{b}} \to \FE{\extnot{f(A)}{b'}}$. As $A \subseteq \extnot{A}{b}$ and $f(A) \subseteq \extnot{f(A)}{b'}$ are $\sigma$-separably $\sigma$-algebraic extensions, the same holds for $A \subseteq \FE{\extnot{A}{b}}$ and $f(A) \subseteq \FE{\extnot{f(A)}{b'}}$, and thus by the tower properties of $\sigma$-separability, $\FE{\extnot{A}{b}} \subseteq K$ and $\FE{\extnot{f(A)}{b'}} \subseteq L$ are $\sigma$-separable extensions. In particular, we get an embedding $f \colon \FE{\extnot{A}{b}} \to \model{L}$. We replace $\FE{\extnot{A}{b}}$ with a residually inversive hull by \cref{inversive-lift}.\\
         We can then repeat this procedure $\omega$-many times to obtain a new countable substructure $A_1$ such that whenever $z \in \O_{A'}$, we have that $\res(z) \in \residue{A_1}$.
         Since $\residue{K}$ is linearly difference closed, for any linear difference equation over $\residue{A}$ we find a solution in $\residue{K}$. By repeating Subcase 3.b, we obtain a new countable substructure $\widetilde{A_1}$ such that whenever $z \in \O_{A'}$, $\res(z) \in \residue{\widetilde{A_1}}$, and furthermore $(\residue{\widetilde{A_1}},\Res{\sigma})$ is linearly difference closed. Note that, as before, we do not necessarily have that whenever $z \in \O_{\extnot{\widetilde{A_1}}{a}}$, $\res(z) \in \residue{\widetilde{A_1}}$. Thus, we repeat this procedure countably many times to produce a chain of extensions $A = A_0 \leq \widetilde{A_1} \leq \widetilde{A_2} \leq \cdots$ such that their union, which is again countable and we call $\widetilde{A_\infty}$, now satisfies that whenever $z \in \O_{\extnot{\widetilde{A_\infty}}{a}}$, $\res(z) \in \residue{\widetilde{A_\infty}}$, and $(\residue{\widetilde{A_\infty}},\Res{\sigma})$ is linearly difference closed. Finally, we rename $\widetilde{A_\infty}$ as $A$. \hfill $\blacksquare^3$
    \end{itemize}
    \par What remains is a valued difference field extension $A \leq A'$ which is $\sigma$-separable and immediate.
    \begin{itemize}
        \item[\text{\textsc{step 4.}}] \textit{The immediate case.}\\
        Replacing $A$ with $\smax{A}$, we may assume that $A$ is $\sigma$-separably $\sigma$-algebraically maximal. By saturation, we may assume that $A'$ is finitely generated over $A$ as a model of $\primary$. Then, by \cref{maclane}, $A'$ is finitely $\sigma$-separably generated, so there is a finite $\sigma$-transcendence basis $\mathcal X \subseteq A'$ such that $\extnot{A}{\mathcal X} \subseteq A'$ is $\sigma$-separably $\sigma$-algebraic. Enumerate $\mathcal X = \{x_1, \cdots x_n\}$. Note that $\extnot{A}{x_1} \subseteq A'$ is $\sigma$-separable by \cref{one-of-them}. We now argue as in \cite[Theorem 7.1]{dor2023contracting} to find an appropriate $y_1 \in \model{L}$.
        \par Using \cref{dependent-defect}, we can find a pseudo-Cauchy sequence $(a_\rho)_\rho \subseteq A$ without pseudolimit in $A$ which is of $\sigma$-transcendental type, and such that $a_\rho \pseudoconverges x_1$. For the sake of the argument, denote by $\mathcal{B} \subseteq L$ the set of pseudolimits of $(a_\rho)_\rho$ in $L$. By saturation, $\mathcal{B}$ is non-empty, and the sequence $(\gamma_\rho)_\rho \subseteq \valuegroup{A} \subseteq \valuegroup{L}$ of radii of the sequence is not cofinal. Pick $\delta \in \valuegroup{L}$ with $\delta > \gamma_\rho$ for every $\rho$. Moreover, pick some $b \in \mathcal{B}$: then $B_\delta(b) \subseteq \mathcal{B}$, thus $\mathcal{B}$ contains an open ball. On the other hand, consider the set $\Lambda \subseteq L$ of elements that are transcendental over $A \otimes_{\sigma(A)} \sigma(L)$: by \cite[Lemma 4.51]{dor2023contracting}, $\Lambda$ is non-empty, and since $\sigma(L)\Lambda \subseteq \Lambda$, it must be dense in $L$. Thus there is $y_1 \in \Lambda \cap \mathcal{B}$. Then, the isomorphism $\extnot{A}{x_1} \to \extnot{A}{y_1}$ over $A$ gives the required embedding of $\extnot{A}{x_1}$ into $L$, with $L$ $\sigma$-separable over $\extnot{A}{y_1}$ by \cite[Proposition 4.50(2)]{dor2023contracting}.
        \par We can now replace $\extnot{A}{x_1}$ with $A_1 \coloneqq \smax{\extnot{A}{x_1}}$; the elements $x_2, \cdots x_n$ remain $\sigma$-algebraically independent over $A_1$ (as this is a $\sigma$-separably $\sigma$-algebraic extension of $\extnot{A}{x_1}$), and thus we can restart the argument by replacing $A$ with $A_1$. We have then produced a tower $A \leq A_1 \leq \cdots \leq A_n \leq K$ of $\sigma$-separable extensions, where each of the $A_i$s is embedded in $L$ so that $L$ is $\sigma$-separable over the image, and such that $A' \subseteq A_1 \cup \cdots \cup A_n$.
    \end{itemize}
    \par This extends $f$ to the required embedding $g$.
\end{proof}

\par We can now deduce the first and most crucial consequence of \cref{embedding-theorem}, namely relative quantifier elimination.

\begin{theorem}\label{relative-qe}
    $\WTHzero$ eliminates quantifiers resplendently relatively to $\RVsort$ and $\valuesort$.
\end{theorem}
\begin{proof}
    Work in the Morleyization of $\WTHzero$ with respect to the sorts $\RVsort$ and $\valuesort$. Let $\model{K}$ and $\model{L}$ be two models, where $\model{K}$ is $\aleph_1$-saturated and $\model{L}$ is $\vert K \vert^+$-saturated. Let $\model{A} \leq \model{K}$ be a countable substructure, and let $f \colon \model{A} \hookrightarrow \model{L}$ be an embedding. Since we are working in the Morleyization, $f_{\RVsort}$ and $f_{\valuesort}$ are elementary, and thus by saturation we can extend them to compatible $\theta_{\RV} \colon \RV_K \hookrightarrow \RV_L$ and $\theta_{\Gamma} \colon \valuegroup{K} \hookrightarrow \valuegroup{L}$. Now, for every $a \in K$, by \cref{embedding-theorem} we find an embedding $g \colon \extnot{A'}{a} \hookrightarrow \model{L}$ extending $f$ such that $g_{\valuesort} = \theta_\Gamma\vert_{\valuesort(\extnot{A}{a})}$, and $g_{\RVsort} = \theta_{\RV}\vert_{\RVsort(\extnot{A}{a})}$.
\end{proof}

\begin{remark}
    In particular, one can work in the reduct of $\Lsep$ where we only consider the sorts $\mainsort$ and $\RVsort$, see \cref{interpret-value}. Then, $\WTHzero$ (seen in this reduct) eliminates quantifiers resplendently down to $\RVsort$.
\end{remark}

\runin{Angular components} Given a valued field $(K,\val)$, we say that a multiplicative group homomorphism $\ac \colon K^\times \to \residue{K}^\times$ is an \emph{angular component} if for every $u \in \O_K^\times$, $\ac(u) = \res(u)$. If $\sigma \in \End(K,\val)$, we say that $\ac$ is \emph{$\sigma$-equivariant} if $\Res{\sigma} \circ \ac = \ac \circ \; \sigma$. In the presence of a $\sigma$-equivariant angular component, we can construct a $\sigma$-equivariant section $s \colon \RV_K^\times \to \residue{K}^\times$, thus the short exact sequence of $\mathbb{Z}[\Val{\sigma}]$-modules
\[
    1 \to \residue{K}^\times \to \RV_K^\times \to \valuegroup{K} \to 0
\]
splits and $\RV_K^\times \cong \residue{K}^\times \times \valuegroup{K}$ along the map $\rv(a) \mapsto (\ac(a),\val(a))$. This allows us to transfer relative quantifier elimination from $\Lsep$ to a more classical three-sorted language with $\residue{K}$ and $\Gamma$, at the cost of adding an angular component.

\begin{definition}
    We let $\Lsepac$ be the three-sorted language with sorts:
    \begin{enumerate}
        \item $\mainsort$, with language $\LKsep$ (see \cref{lksep}),
        \item $\residuesort$, with language $\Lringsigma = \{+,\cdot,-,0,1,\Res{\sigma}\}$,
        \item $\valuesort$, with language $\Loagsigma \coloneqq  \{+,\leq,0,\infty,\Val{\sigma}\}$,
    \end{enumerate}
    and connecting functions $\ac\colon \mainsort \to \residuesort$ and $\val\colon \mainsort \to \valuesort$. We let $\WTHzeroac$ be the theory of non-inversive, equicharacteristic zero, weakly $\sigma$-henselian valued difference fields that are models of $\primary$, endowed with a $\sigma$-equivariant angular component.
\end{definition}

We now sketch how to deduce relative quantifier elimination for $\WTHzeroac$ in $\Lsepac$ from relative quantifier elimination  for $\WTHzero$ in $\Lsep$.
\begin{enumerate}
    \item Consider the expansion $\Lsep \subseteq \Lsepenrich$ where the sort $\RVsort$ is enriched with a map $\ac_{\rv}\colon \RVsort \to \RVsort$, to be interpreted as the map induced by $\ac$ along the diagram
    \[\begin{tikzcd}
    	{\mainsort^\times} && {\residuesort^\times} \\
    	& \RVsort
    	\arrow["\ac", from=1-1, to=1-3]
    	\arrow["\rv"', from=1-1, to=2-2]
    	\arrow["{\ac_{\rv}}"', from=2-2, to=1-3]
    \end{tikzcd}\]
    Then, resplendency of the relative quantifier elimination result means that the rephrasing $\WTHzero^+$ of $\WTHzero$ in $\Lsepenrich$ still eliminates quantifiers relatively to $\RVsort$ and $\valuesort$.
    \item We now observe that, in models of $\WTHzero^+$, $\residuesort^\times$ and $\ac$ are definable in $\Lsepenrich$. Namely, $y \in \residuesort^\times$ if and only if $\val_{\rv}(y) = 0$, and $\ac(x) = z$ if and only if $\ac_{\rv}(\rv(x)) = z$. In the other direction, $\RVsort$ and $\ac_{\rv}$ are $\mainsort$-quantifier-free interpretable in $\Lsepac$. Indeed, in the presence of $\ac$ one has the identification $\RVsort^\times = \residuesort^\times \times \valuesort$, where $\ac_{\rv}$ is the projection on the first coordinate. Thus, one can transfer relative quantifier elimination between these two languages.
\end{enumerate}

\begin{theorem}\label{relative-qe-ac}
    $\WTHzeroac$ eliminates quantifiers resplendently relatively to $\residuesort$ and $\valuesort$.
\end{theorem}

\begin{theorem}
    Let $\model{K}$ be a model of $\WTHzero$. Then $(\RV_K,\Rv{\sigma})$ and $(\valuegroup{K},\Val{\sigma})$ are stably embedded, with induced structures given respectively by $\L_{\RVsort}$ and $\Loagsigma$. Analogously, if $\model{K}$ is a model of $\WTHzeroac$, then $(\residue{K},\Res{\sigma})$ and $(\valuegroup{K},\Val{\sigma})$ are stably embedded, with induced structures given respectively by $\Lringsigma$ and $\Loagsigma$.
\end{theorem}

\runin{Lifting the residue field} We now turn to an enriched version of $\Lsep$. In many applications, it becomes useful to see the residue difference field $(\residue{K},\Res{\sigma})$ as a difference subfield of the valued difference field $(K,v,\sigma)$. For example, this is successfully exploited in \cite{hils2024langweil}. First, we argue that such lifts exist.

\begin{lemma}
    Suppose $(K,v,\sigma)$ is equicharacteristic zero, weakly $\sigma$-henselian, and $\aleph_0$-saturated. Let $F \subseteq \residue{K}$ be an inversive difference subfield, and suppose there exists an embedding $\eta \colon F \hookrightarrow \O^{\times}_{\core{K}}$ such that $\res \circ \; \eta = \id$. Given any $\alpha \in \residue{K}$, there is an embedding $\widetilde{\eta} \colon \ext{F}{\alpha} \hookrightarrow \O^\times_{\core{K}}$, extending $\eta$ and such that $\res \circ \; \widetilde{\eta} = \id$.
\end{lemma}

\[\begin{tikzcd}
	{\residue{K}} && {\core{K}} \\
	\\
	F
	\arrow["{\widetilde{\eta}}"', color={lightgray}, hook, from=1-1, to=1-3]
	\arrow["\res"', curve={height=18pt}, from=1-3, to=1-1]
	\arrow[hook, from=3-1, to=1-1]
	\arrow["\eta", hook, from=3-1, to=1-3]
\end{tikzcd}\]

\begin{proof}
    Let $a \in \O_{\core{K}}^\times$ be such that $\res(a) = \alpha$.
    If for all $g(X) \in \polring{\eta(F)}$ we have that $\val(g(a)) = 0$, then $a$ is $\sigma$-transcendental over $\eta(F)$ and $\alpha$ is $\sigma$-transcendental over $F$, so $\eta$ extends to a difference field isomorphism $\ext{F}{\alpha} \to \ext{\eta(F)}{a} \subseteq \O_{\core{K}}^\times$.
    Suppose now that there is some $g(X) \in \polring{\eta(F)}$ such that $\val(g(a)) > 0$, and suppose that it is of minimal complexity such. Then, $g(X)$ is $\sigma$-henselian at $a$, so by strong $\sigma$-henselianity of $\core{K}$ (\cref{core-dense}) we find $b \in \O_{\core{K}}$ such that $g(b) = 0$ and $\res(b) = \alpha$. In particular, then, $\eta$ extends to a difference field isomorphism $\ext{F}{\alpha} \to \ext{\eta(F)}{b} \subseteq \O_{\core{K}}^\times$.
\end{proof}

\begin{corollary}
    Suppose $(K,v,\sigma)$ is equicharacteristic zero, weakly $\sigma$-henselian, and $\aleph_0$-saturated. Then there is a difference subfield $F \subseteq \O_{\core{K}}^\times$ such that $\res\vert_{F} \colon F \to \residue{K}$ is a difference field isomorphism.
\end{corollary}

We say that a map $s \colon \valuegroup{K} \to K$ is a \emph{section} of the valuation if for all $x \in \valuegroup{K}$, $\val(s(x)) = x$. We say that $s$ is \emph{$\sigma$-equivariant} if $s \circ \sigma = \Val{\sigma}^{-1} \circ s$.

\begin{remark}
    Note that if $s$ is a $\sigma$-equivariant section of the valuation, then 
    \[
        \ac(x) \coloneqq \res\left(\frac{x}{s(\val(x))}\right)
    \]
    is a $\sigma$-equivariant angular component.
\end{remark}

\begin{definition}
        We let $\Lseclift$ be the three-sorted language with sorts:
        \begin{enumerate}
            \item $\mainsort$, with language $\LKsep$ (see \cref{lksep}),
            \item $\residuesort$, with language $\Lringsigma = \{+,\cdot,-,0,1,\Res{\sigma}\}$,
            \item $\valuesort$, with language $\Loagsigma \coloneqq  \{+,\leq,0,\infty,\Val{\sigma}\}$,
        \end{enumerate}
        and connecting functions $s \colon \valuesort \to \mainsort$, $\val\colon \mainsort \to \valuesort$, and $\iota \colon \residuesort \to \mainsort$. We let $\WTHzerosl$ be the theory of non-inversive, equicharacteristic zero, weakly $\sigma$-henselian valued difference fields that are models of $\primary$, and where $s$ is interpreted as a $\sigma$-equivariant section of the valuation, and $\iota$ as an embedding of $\residuesort$ in $\mainsort$.
\end{definition}

Arguing as in \cite{kesting}, we can obtain the corresponding results in this language.

\begin{theorem}\label{relative-qe-sl}
    $\WTHzerosl$ eliminates quantifiers resplendently relatively to $\residuesort$ and $\valuesort$.
\end{theorem}

\runin{Ax-Kochen/Ershov} We now establish transfer theorems for $\WTHzero$, $\WTHzeroac$, and $\WTHzerosl$, using \cref{relative-qe}, \cref{relative-qe-ac}, and \cref{relative-qe-sl}.

\begin{theorem}\label{ake}
    Let $\model{K}$ and $\model{L}$ be models of $\WTHzero$. Then,
    \[
        \model{K} \equiv \model{L} \iff (\RV_K,\Rv{\sigma}) \equiv (\RV_L,\Rv{\sigma}).
    \]
    If $\model{K} \leq \model{L}$, then
    \[
        \model{K} \preceq \model{L} \iff (\RV_K,\Rv{\sigma}) \preceq (\RV_L,\Rv{\sigma}).
    \]
    Analogously, let $\model{K}$ and $\model{L}$ be models of $\WTHzeroac$ or $\WTHzerosl$. Then,
    \[
        \model{K} \equiv \model{L} \iff ((\residue{K},\Res{\sigma}) \equiv (\residue{L},\Res{\sigma}) \; \text{and} \; (\valuegroup{K},\Val{\sigma}) \equiv (\valuegroup{L},\Val{\sigma})).
    \]
    If $\model{K} \leq \model{L}$, then
    \[
        \model{K} \preceq \model{L} \iff ((\residue{K},\Res{\sigma}) \preceq (\residue{L},\Res{\sigma}) \; \text{and} \; (\valuegroup{K},\Val{\sigma}) \preceq (\valuegroup{L},\Val{\sigma})).
    \]
\end{theorem}
\begin{proof}
    The proofs are agnostic to which one of the three languages we work in.
    \par $(\equiv)$: this follows from \cref{relative-qe}, since $(\mathbb{Q},\val_{\mathrm{triv}},\mathrm{id}_{\mathbb{Q}})$ is a common substructure between any two models of $\WTHzero$.
    \par $(\preceq)$: This is a direct consequence of relative quantifier elimination.
\end{proof}

\begin{proposition}\label{general-existential}
    Let $\model{K}$ and $\model{L}$ be two models of $\WTHzero$ and $\model{A} \leq \model{K},\model{L}$ be a common substructure. Then,
    \[
        \Th^\exists_{\Lsorts(\RVsort(\model{A})\cup\valuesort(\model{A}))}(\RV_K,\valuegroup{K}) \subseteq  \Th^\exists_{\Lsorts(\RVsort(\model{A})\cup\valuesort(\model{A}))}(\RV_L,\valuegroup{L})
    \]
    if and only if
    \[
        \Th^\exists_{\Lsep(\model{A})}(\model{K}) \subseteq \Th^\exists_{\Lsep(\model{A})}(\model{L}).
    \]
    Analogously, let $\model{K}$ and $\model{L}$ be models of $\WTHzeroac$ or $\WTHzerosl$. Then,
    \[
        \Th^\exists_{\Lsorts(\residuesort(\model{A})\cup\valuesort(\model{A}))}(\residue{K},\valuegroup{K}) \subseteq  \Th^\exists_{\Lsorts(\residuesort(\model{A})\cup\valuesort(\model{A}))}(\residue{L},\valuegroup{L})
    \]
    if and only if
    \[
        \Th^\exists_{\L(\model{A})}(\model{K}) \subseteq \Th^\exists_{\L(\model{A})}(\model{L})
    \]
    for the appropriate $\L \in \{\Lsepac,\Lseclift\}$.
\end{proposition}
\begin{proof}
    We give the argument for $\Lsep$.
    One direction is clear. For the other, without loss of generality we may assume that $\model{K}$ and $\model{L}$ are $\aleph_1$-saturated and $\model{A}$ is countable. Suppose that $\Th^\exists_{\Lsep(\model{A})}(\model{K}) \not\subseteq \Th^\exists_{\Lsep(\model{A})}(\model{L})$. Let this be witnessed by some quantifier-free formula $\varphi(X)$ with parameters in $A$, and let $b \in K$ be such that $K \vDash \varphi(b)$, but $L \not\vDash \exists X \varphi(X)$. By assumption, we have compatible embeddings $\theta_\RV \colon \RV_K \to \RV_L$ and $\theta_\Gamma \colon \valuegroup{K} \to \valuegroup{L}$. Then, \cref{embedding-theorem} yields an embedding $g\colon \model{C} \to \model{L}$ of some countable substructure $A \leq \model{C} \leq \model{L}$ containing $b$, so $L \vDash \varphi(g(b))$. This is a contradiction, since then in particular $L \vDash \exists X \varphi(X)$.
\end{proof}

This implies two more recognizable results.

\begin{theorem}\label{existential-ake}
    Let $\model{K}$ and $\model{L}$ be two models of $\WTHzero$, with $\model{K} \leq \model{L}$. Then, $\model{K} \preceq_\exists \model{L}$ if and only if $(\RV_K,\Rv{\sigma}) \preceq_\exists (\RV_L,\Rv{\sigma})$. Analogously, let $\model{K}$ and $\model{L}$ be two models of $\WTHzeroac$ or $\WTHzerosl$, with $\model{K} \leq \model{L}$. Then, $\model{K} \preceq_\exists \model{L}$ if and only if $(\residue{K},\Res{\sigma}) \preceq_\exists (\residue{L},\Res{\sigma})$ and $(\valuegroup{K},\Val{\sigma}) \preceq_\exists (\valuegroup{L},\Val{\sigma})$.
\end{theorem}
\begin{proof}
    One direction is clear. For the other, apply \cref{general-existential} with $A = K$.
\end{proof}

\begin{theorem}\label{existential-ake-2}
    Let $\model{K}$ and $\model{L}$ be two models of $\WTHzero$. Then, $\operatorname{Th}^\exists_{\Lsep}(\model{K}) \subseteq \operatorname{Th}^\exists_{\Lsep}(\model{L})$ if and only if $\operatorname{Th}^\exists(\RV_K,\Rv{\sigma}) \subseteq \operatorname{Th}^\exists(\RV_L,\Rv{\sigma})$. Analogously, let $\model{K}$ and $\model{L}$ be two models of $\WTHzeroac$ or $\WTHzerosl$. Then, $\operatorname{Th}^\exists_{\L}(\model{K}) \subseteq \operatorname{Th}^\exists_{\L}(\model{L})$ if and only if $\operatorname{Th}^\exists(\valuegroup{K},\Val{\sigma}) \subseteq \operatorname{Th}^\exists(\valuegroup{L},\Val{\sigma})$ and $\operatorname{Th}^\exists(\residue{K},\Res{\sigma}) \subseteq \operatorname{Th}^\exists(\residue{L},\Res{\sigma})$, for the appropriate $\L \in \{\Lsepac,\Lseclift\}$.
\end{theorem}
\begin{proof}
    One direction is clear. For the other, apply \cref{general-existential} with $A = \mathbb Q$.
\end{proof}

%% file: advanced-mt.tex
We follow the strategy developed by Chernikov and Hils in \cite{chernikov2014valued}. We first show that dense pairs of algebraically closed valued fields are $\mathrm{NIP}$, a folklore fact for which we couldn't find a proof in the literature; we include one for the convenience of the reader, using the same strategy deployed by Delon in \cite{delon2types} to show that separably algebraically closed valued fields are $\mathrm{NIP}$.

\begin{definition}
    we let $\LDP$ be the three-sorted Denef-Pas language:
    \begin{enumerate}
        \item $\mainsort$, with language $\{+,\cdot,-,0,1\}$,
        \item $\residuesort$, with language $\{+,\cdot,-,0,1\}$,
        \item $\valuesort$, with language $\{+,\leq,0,\infty\}$,
    \end{enumerate}
    and connecting functions $\ac\colon \mainsort \to \residuesort$ and $\val\colon \mainsort \to \valuesort$. We let $\LDP^*$ be an expansion of $\LDP$ where, on the sort $\mainsort$, we add a unary predicate $P(x)$, and countably many functions $(\lambda_n^i)_{n \in \mathbb N, 1 \leq i \leq n}$, where each $\lambda_n^i$ is $(n+1)$-ary. Denote by $\Lringpairs$ the resulting language given by $\{+,\cdot,-,0,1\} \cup \{P(x)\} \cup \{\lambda_n^i \mid n \in \mathbb{N}, 1 \leq i \leq n\}$.
\end{definition}

\begin{definition}
    We let $\mathrm{PPairs}$ be the $\LDPlambda$-theory of proper pairs $(K,P(K))$ of valued fields, where $P(K)$ is relatively algebraically closed and the $\lambda$-functions are interpreted as giving the coefficients for $P$-linear independence. We let $\denseacvf$ be the $\LDPlambda$-theory of proper dense pairs of algebraically closed valued fields.
\end{definition}

\begin{remark}
    Note that, when we interpret a valued difference field $(K,\val,\sigma)$ as an $\LDPlambda$-structure with $P(K) = \sigma(K)$, the $\lambda$-functions from $\LDPlambda$ do not coincide with the $\lambda$-functions introduced in \cref{lambda-sigma}. Namely, the $\lambda$-functions from $\LDPlambda$ output the coefficient of a $P(K)$-linear combination, an element of $P(K) = \sigma(K)$, whereas the $\lambda$-functions in \cref{lambda-sigma} output $\sigma^{-1}$ of such a coefficient, so an element of $K$. The functions are however interdefinable (and in fact, both $0$-definable in the respective reducts).
\end{remark}
    
\begin{proposition}\label{acvfpairisnip}
    $\denseacvf$ is $\mathrm{NIP}$.
\end{proposition}
\begin{proof}
    Let $(M,P(M)) \vDash \denseacvf$ and let $(M,P(M)) \preceq (\mathcal{U}, P(\mathcal{U}))$ be a monster model. We show that for any $p \in S(M)$ there are at most $2^{\vert M \vert}$ global coheirs of $p$, i.e. $p \subseteq q \in S(\mathcal{U})$ with $q$ finitely satisfiable in $M$.
    \par Indeed, take any $a \vDash p$ and let $N \succeq M$ be such that $a \in N$ and $\vert N \vert = \vert M \vert$. It is enough to show that $p' = \operatorname{tp}(N/M)$ has at most $2^{\vert M \vert}$ global coheirs.
    \par  Since pairs of algebraically closed fields are stable, $\operatorname{tp}_{\Lringpairs}(N/M)$ has a unique global coheir, call it $r$. If $N' \vDash r$, then $N'$ is linearly disjoint from $\mathcal{U}$ over $M$, and the $\lambda$-functions are uniquely determined on the compositum $N'\mathcal{U}$. In other words, the $\Lringpairs$-structure on $N'\mathcal{U}$ is uniquely determined, and the number of global coheirs of $\operatorname{tp}(N/M)$ is bounded by the number of global coheirs of the $\mathrm{ACVF}$-type of $N$ over $M$. Since $\mathrm{ACVF}$ is $\mathrm{NIP}$ by \cite{delon2types}, there are at most $2^{\vert N \vert} = 2^{\vert M \vert}$ many such.
\end{proof}

\begin{lemma}\label{embed-in-acvf}
    Any model of $\mathrm{PPairs}$ embeds in a model of $\mathrm{ACVF}_{\mathrm{dense}}$.
\end{lemma}
\begin{proof}
    Let $(K,P(K))$ be a model of $\mathrm{PPairs}$. Without loss of generality, we may assume that $K$ is algebraically closed, and thus so is $P(K)$. We can now make $P(K)$ dense in $K$ by a chain construction: for every ball $B \subseteq K$ with $B \cap P(K) = \emptyset$, let $t$ be a generic of $B$ in the predicate. In particular, $K \subseteq K(t)$ and $P(K) \subseteq P(K)(t)$ are Gauss extensions. Then, $(K,P(K)) \subseteq (K(t),P(K)(t))$ is an extension in $\LDPlambda$. We then once again move to the algebraic closures. By iterating this construction, the limit is a model of $\mathrm{ACVF}_{\mathrm{dense}}$.
\end{proof}

\begin{lemma}\label{lemma-nip}
    Let $\varphi(x,y)$ be a quantifier-free $\Lsepac$-formula. Then, modulo $\Tvdfzero$, it is $\mathrm{NIP}$.
\end{lemma}
\begin{proof}
    We proceed as in \cite[Lemma 4.2]{chernikov2014valued}. We can write $\varphi(x,y)$ as 
    \[
        \psi(x,\sigma(x),\dots \sigma^n(x), y, \sigma(y), \dots \sigma^n(y)),
    \]
    where $\psi(x_0,x_1,\dots x_n, y_0,y_1, \dots y_n)$ is a quantifier-free $\LDPlambda$-formula.
    \par Without loss of generality, we may assume that we are working with a non-surjective endomorphism; otherwise, the result is \cite[Lemma 4.2]{chernikov2014valued}.
    Now, by \cref{embed-in-acvf}, to check that $\psi(\overline{x},\overline{y})$ is $\mathrm{NIP}$ in $\mathrm{PPairs}$, it is enough to check if $\psi(\overline{x},\overline{y})$ is $\mathrm{NIP}$ in $\denseacvf$. We conclude by \cref{acvfpairisnip}.
\end{proof}

\begin{theorem}\label{ntp2transfer}
    Suppose $\model{K}$ is a model of $\WTHzeroac$. Then, $\model{K}$ is $\mathrm{NTP}_2$ in $\Lsepac$ if and only if $(\residue{K},\Res{\sigma})$ is $\mathrm{NTP}_2$ in $\Lringsigma$ and $(\valuegroup{K},\Val{\sigma})$ is $\mathrm{NTP}_2$ in $\Loagsigma$.
\end{theorem}
\begin{proof}
    The proof of \cite[Theorem 4.1]{chernikov2014valued} can be followed verbatim, once one replaces their Lemma 4.2 with \cref{lemma-nip}. Note that while they work in the multiplicative case, this is really not necessary, as the same results hold in the general case by \cite{onay2013quantifier}.
\end{proof}
    
    One can prove the same result in $\Lsep$ and $\Lseclift$:
    
\begin{theorem}\label{ntp2-rv}
    Suppose $\model{K}$ is a model of $\WTHzero$. Then, $\model{K}$ is $\mathrm{NTP}_2$ in $\Lsep$ if and only if $(\RV_K,\Rv{\sigma})$ is $\mathrm{NTP}_2$ in $\L_{\RVsort}$.
\end{theorem}

\begin{theorem}\label{ntp2-sl}
    Suppose $\model{K}$ is a model of $\WTHzerosl$. Then, $\model{K}$ is $\mathrm{NTP}_2$ in $\Lseclift$ if and only if $(\residue{K},\Res{\sigma})$ is $\mathrm{NTP}_2$ in $\Lringsigma$ and $(\valuegroup{K},\Val{\sigma})$ is $\mathrm{NTP}_2$ in $\Loagsigma$.
\end{theorem}

%% file: main.bbl
\begin{thebibliography}{HHYZ24}

\bibitem[AK65]{ax1965diophantine}
James Ax and Simon Kochen.
\newblock {Diophantine problems over local fields I}.
\newblock {\em Amer. J. Math.}, 87(3):605--630, 1965.

\bibitem[AvdD11]{azgin2011elementary}
Salih Azgın and Lou van~den Dries.
\newblock Elementary theory of valued fields with a valuation-preserving automorphism.
\newblock {\em J. Inst. Math. Jussieu}, 10(1):1--35, 2011.

\bibitem[Azg10]{azgin2010valued}
Salih Azgın.
\newblock Valued fields with contractive automorphism and {Kaplansky} fields.
\newblock {\em J. Algebra}, 324(10):2757--2785, 2010.

\bibitem[BMS07]{belair2007model}
Luc B{\'e}lair, Angus Macintyre, and Thomas Scanlon.
\newblock {Model theory of the Frobenius on the Witt vectors}.
\newblock {\em Amer. J. Math.}, 129(3):665--721, 2007.

\bibitem[CH04]{chatzidakis2004model}
Zo{\'e} Chatzidakis and Ehud Hrushovski.
\newblock Model theory of endomorphisms of separably closed fields.
\newblock {\em J. Algebra}, 281(2):567--603, 2004.

\bibitem[CH14]{chernikov2014valued}
Artem Chernikov and Martin Hils.
\newblock {Valued difference fields and $\mathrm{NTP}_2$}.
\newblock {\em Israel J. Math.}, 204(1):299--327, 2014.

\bibitem[Del79]{delon2types}
Fran{\c{c}}oise Delon.
\newblock Types sur $\mathbb{C}(\!(x)\!)$.
\newblock {\em Groupe d'{\'e}tude de th{\'e}ories stables}, 2:1--29, 1978-1979.

\bibitem[DH22]{dor2022specialization}
Yuval Dor and Ehud Hrushovski.
\newblock {Specialization of Difference Equations and High Frobenius Powers}.
\newblock {\em arXiv preprint \href{https://arxiv.org/abs/2212.05366}{arXiv:2212.05366}}, 2022.

\bibitem[DH23]{dor2023contracting}
Yuval Dor and Yatir Halevi.
\newblock {Contracting Endomorphisms of Valued Fields}.
\newblock {\em arXiv preprint \href{https://arxiv.org/abs/2305.18963}{arXiv:2305.18963}}, 2023.

\bibitem[DO15]{onay2013quantifier}
Salih Durhan and G{\"o}nen{\c{c}} Onay.
\newblock Quantifier elimination for valued fields equipped with an automorphism.
\newblock {\em Selecta Mathematica}, 21:1177--1201, 2015.

\bibitem[Ers65]{ershov1965elementary}
Yurii Ershov.
\newblock On the elementary theory of maximal normed fields.
\newblock In {\em Dokl. Akad.}, volume 165, pages 21--23. Russian Academy of Sciences, 1965.

\bibitem[HHYZ24]{hils2024langweil}
Martin Hils, Ehud Hrushovski, Jinhe Ye, and Tingxiang Zou.
\newblock {Lang-Weil Type Estimates in Finite Difference Fields}.
\newblock {\em arXiv preprint \href{https://arxiv.org/abs/2406.00880}{arXiv:2406.00880}}, 2024.

\bibitem[HK23]{halevi2023saturated}
Yatir Halevi and Itay Kaplan.
\newblock Saturated models for the working model theorist.
\newblock {\em Bulletin of Symbolic Logic}, 29(2):163--169, 2023.

\bibitem[Hru]{nonstandardfrob}
Ehud Hrushovski.
\newblock {The Elementary Theory of the Frobenius Automorphism}.
\newblock {\em preprint, available at \href{http://www.ma.huji.ac.il/~ehud/FROB.pdf}{http://www.ma.huji.ac.il/~ehud/FROB.pdf}}.

\bibitem[Kes21]{kesting}
Christoph Kesting.
\newblock {Tameness Properties in Multiplicative Valued Difference Fields with Lift and Section}.
\newblock {\em Thesis for the MSc in Mathematics at the University of Münster, to appear}, 2021+.

\bibitem[KR23]{kuhlmann2023valuation}
Franz-Viktor Kuhlmann and Anna Rzepka.
\newblock The valuation theory of deeply ramified fields and its connection with defect extensions.
\newblock {\em Trans. Amer. Math. Soc.}, 376(04):2693--2738, 2023.

\bibitem[Rid17]{rideau2017some}
Silvain Rideau.
\newblock Some properties of analytic difference valued fields.
\newblock {\em J. Inst. Math. Jussieu}, 16(3):447--499, 2017.

\bibitem[SV21]{shuddhodan2021hrushovski}
Kadattur~V Shuddhodan and Yakov Varshavsky.
\newblock {The Hrushovski-Lang-Weil estimates}.
\newblock {\em arXiv preprint \href{https://arxiv.org/abs/2106.10682}{arXiv:2106.10682}}, 2021.

\bibitem[Wei94]{weibel1994introduction}
Charles~A Weibel.
\newblock {\em An introduction to homological algebra}.
\newblock Number~38. Cambridge University Press, 1994.

\end{thebibliography}
